\documentclass[11pt,oneside]{amsart}
\usepackage{amsmath,amsfonts,amsgen,amstext,amsbsy,amsopn,amsthm, amssymb}
\usepackage{multirow, verbatim, cancel, mathabx}
\usepackage{bbold}
\usepackage{enumitem}
\usepackage[T1]{fontenc}
\usepackage[colorlinks=true,hyperindex=true]{hyperref}
\usepackage{tikz}
\usetikzlibrary{calc}
\usepackage{verbatim}
\usetikzlibrary{arrows,chains,matrix,positioning,scopes}
\usepackage{mathrsfs}
\makeatletter

\usepackage{pgfplotstable}

\usepackage{pgfplots}
\usepackage{setspace}

\numberwithin{equation}{section}

\newcounter{smallarabics}

\newcounter{smallroman}
\newenvironment{romanenumerate}
{\begin{list}{{\normalfont\textrm{(\roman{smallroman})}}}
  {\usecounter{smallroman}\setlength{\itemindent}{0cm}
   \setlength{\leftmargin}{5ex}\setlength{\labelwidth}{4ex}
   \setlength{\topsep}{0.75\parsep}\setlength{\partopsep}{0ex}
   \setlength{\itemsep}{0ex}}}
{\end{list}}

\newcommand{\ben}{\begin{romanenumerate}}  
\newcommand{\een}{\end{romanenumerate}}  

\newtheorem{theoreme}{theorem }[section]
\newtheorem{theorem}[theoreme]{Theorem}
\newtheorem{proposition}[theoreme]{Proposition}
\newtheorem{Lemma}[theoreme]{Lemma}

\newtheorem{remark}{Remark}[section]

\newcommand{\Pp}{P^{\perp}}

\newcommand{\tH}{\theta(H)}

\newcommand{\eH}{\eta(H)}
\newcommand{\eD}{\eta(\Delta)}
\newcommand{\pp}{\frac{\partial \tilde{\varphi}}{\partial \overline{z}}(z)}

\newcommand{\A}{(z-A/R)^{-1}}
\newcommand{\JapA}{\Big \langle \frac{A}{R} \Big \rangle }
\newcommand{\const}{\frac{\i}{2\pi}}
\newcommand{\dz}{dz\wedge d\overline{z}}

\renewcommand\leq\varleq
\renewcommand\geq\vargeq
\newcommand{\sign}{\mathrm{sign}}

\newcommand{\1}{\mathbb 1}

\renewcommand{\proof}{\noindent \emph{Proof. }}
 
 \newcommand{\R}{\mathbb{R}}
 \newcommand{\N}{\mathbb{N}}
\newcommand{\Z}{\mathbb{Z}} \newcommand{\C}{\mathbb{C}}
\renewcommand{\i}{\text{i}}

\DeclareMathOperator*{\slim}{s-lim}

\usepackage{dcolumn}
\usepackage{tabu}

\newcolumntype{A}{D{.}{.}{2.3}}

\makeatletter
      \def\@setcopyright{}
      \def\serieslogo@{}
      \makeatother
\begin{document}

\author{Mandich, Marc-Adrien}
   \address{Institut de Math\'ematiques de Bordeaux, 351 cours de la Lib\'{e}ration, F 33405 Talence, France}
   \email{marc-adrien.mandich@u-bordeaux.fr}


   \title[LAP for discrete Wigner-von Neumann Operators]{The Limiting Absorption Principle \\ for the discrete Wigner-von Neumann Operator}

   \begin{abstract}
     We apply weighted Mourre commutator theory to prove the limiting absorption principle for the discrete Schr\"{o}dinger operator perturbed by the sum of a Wigner-von Neumann and long-range type potential. In particular, this implies a new result concerning the absolutely continuous spectrum for these operators even for the one-dimensional operator. We show that methods of classical Mourre theory based on differential inequalities and on the generator of dilation cannot apply to the mentionned Schr\"{o}dinger operators. 
   \end{abstract}

%



   \maketitle

\section{Introduction}

The spectral theory of discrete Schr{\"o}dinger operators has received much attention in the past few decades. The absolutely continuous spectrum is important because it allows to describe the quantum dynamics of a system. The limiting absorption principle (LAP) plays a profound role in spectral and scattering theory, in particular, it implies the existence of purely absolutely continuous spectrum. The LAP has been derived for a wide class of potentials, including the Wigner-von Neumann potential (cf.\ \cite{NW}, \cite{DMR}, \cite{RT1}, \cite{RT2}, \cite{MS} and \cite{EKT} to name a few), but only recently has the sum of a Wigner-von Neumann and long-range potential been studied in the continuous setting (cf.\ \cite{GJ2}). The LAP has not been studied for the discrete Wigner-von Neumann operator. On the other hand, the absolutely continuous spectrum of the one-dimensional Wigner-von Neumann operator plus a potential $V \in \ell^1(\Z)$ has already been studied, both in the discrete and continuous setting (cf.\ \cite {Si}, \cite{JS}, \cite{KN}, \cite{NS}, \cite{KS}). To our knowledge, the question of the absolutely continous spectrum of the discrete Wigner-von Neumann operator plus a long-range potential $V$ has not been studied in any dimension. In this paper we study the sum of a Wigner-von Neumann and long-range potential in the discrete setting which we now describe.

The configuration space is the multi-dimensional lattice $\Z^d$ for some integer $d \geqslant 1$. For a multi-index $n = (n_1,...,n_d) \in \Z^d$ we set $|n|^2 := n_1^2 + ...+n_d^2$. Consider the Hilbert space $\mathcal{H} := \ell^2(\Z^d)$ of square summable sequences $u = (u(n))_{n \in \Z^d}$. The discrete Schr{\"o}dinger operator 
\begin{equation}
\label{low1}
H := \Delta + W + V
\end{equation}
acts on $\mathcal{H}$, where $\Delta$ is the discrete Laplacian operator defined by
\begin{equation*}
(\Delta u)(n) := \sum \limits_{\substack{m \in \Z^d \\ |n-m|=1}} (u(n) - u(m)), \quad \text{for all} \ n \in \Z^d \ \text{and} \ u \in \mathcal{H}, 
\end{equation*}
$W$ is the Wigner-von Neumann potential defined by
\begin{equation}
\label{Wiggy8}
(Wu)(n) := \frac{q\sin(k(n_1+...+n_d))}{|n|}u(n), \quad \text{for all} \ n \in \Z^d \ \text{and} \ u \in \mathcal{H},
\end{equation}
with $q \in \R\setminus\{0\}$ and $k \in \mathcal{T} := \R \setminus \pi \Z \ (\text{mod} \ 2\pi)$, and $V$ is a multiplication operator by a real-valued sequence $(V(n))_{n \in \Z^d}$:
\begin{equation*}
(V u)(n) := V(n)u(n), \quad \text{for all} \ n \in \Z^d \ \text{and} \ u \in \mathcal{H}.
\end{equation*}
We will also investigate the following variation on the Wigner-von Neumann potential:
\begin{equation}
\label{Wdefine}
(W' u)(n) :=  \left(\prod_{i=1}^d \frac{q_i \sin(k_i n_i)}{n_i}\right)u(n), \quad \text{for all} \ n \in \Z^d \ \text{and} \ u \in \mathcal{H},
\end{equation}
with $q=(q_i)_{i=1}^d \in (\R\setminus \{0\})^d$ and $k=(k_i)_{i=1}^d \in \mathcal{T}^d$. In this case, we shall denote $H' := \Delta + W' +V$. In the definition of $W$ and $W'$, it is understood that $\sin(0)/0 := 1$. The potential $V$ will be of long-range type, hence a compact operator, but we postpone the characterization of its exact decay properties.
Using the Fourier transform $\mathcal{F} : \mathcal{H} \to L^2([-\pi,\pi]^d,d\xi)$, $\xi=(\xi_1,...,\xi_d)$, we get 
\begin{equation}
(\mathcal{F} \Delta \mathcal{F}^{-1}f)(\xi) = f(\xi) \sum \limits_{i=1}^d (2-2\cos(\xi_i)), \quad \ \text{where} \quad  (\mathcal{F} u)(\xi) :=  \sum \limits_{n \in \Z^d} u(n) e^{\i n \cdot \xi}(2\pi)^{-d/2}.
\label{shammy}
\end{equation}
This shows that $\Delta$ is a bounded self-adjoint operator on $\mathcal{H}$, and that $\sigma(\Delta) = \sigma_{\text{ac}}(\Delta) = [0,4d]$. The operators $H$ and $H'$ are compact perturbations of $\Delta$ and so $\sigma_{\text{ess}}(H) = \sigma_{\text{ess}}(H') = [0,4d]$.

When $V \equiv 0$, we are left with the Wigner-von Neumann potential only, the example of a Schr{\"o}dinger operator with an eigenvalue embedded in the absolutely continuous spectrum (cf.\ \cite{NW}, \cite{RS4}). In the continuous setting it has been shown that the $1d$ Schr{\"o}dinger operator 
\begin{equation*}
-\frac{d^2}{dx^2}+\frac{q\sin(kx)}{x}+O(x^{-2})
\end{equation*}
covers the interval $[0,\infty)$ with absolutely continuous spectrum and may produce exactly one eigenvalue with positive energy. In the discrete setting the $1d$  Schr{\"o}dinger operator $\Delta + W$ covers the interval $(0,4)$ with absolutely continuous spectrum due to the fact that $W \in \ell^2(\Z)$ (cf.\ \cite{DK}), and it has been shown (cf.\ \cite{JS}, \cite{Si}) that there are two critical points located at 
\begin{equation}
\label{yummy}
E_{\pm}(k) := 2 \pm 2\cos\left(k/2\right)
\end{equation} 
which may be half-bound states or eigenvalues. If $V \in \ell^1(\Z)$, the spectrum of $H := \Delta + W + V$ is purely absolutely continuous on $(0,4) \setminus \{E_{\pm}(k)\}$ (cf.\ \cite{JS}). The works \cite {Si}, \cite{JS}, \cite{KN}, \cite{NS}, and \cite{KS} are concerned with the asymptotics of the generalized eigenvectors of $H := \Delta + P + W + V$, where $P$ is periodic, $W$ is the Wigner-von Neumann potential and $V \in \ell^1(\Z)$. 
 
We fix some notation. Let $S := (S_1,...,S_d)$ where, for $1 \leqslant i \leqslant d$, $S_i$ is the shift operator 
\begin{equation}
(S_iu)(n) := u(n_1,...,n_i-1,...,n_d), \quad \text{for all} \ n \in \Z^d \ \text{and} \ u \in \mathcal{H}.
\label{ShiftOperator}
\end{equation}  
We denote by $\tau_i V$ (resp. $\tau_i^* V$) the operator of multiplication acting by 
\begin{equation*}
[(\tau_i V)u](n) := V(n_1,...,n_i-1,...n_d)u(n) \  \left(\text{resp.}  \ [(\tau_i^* V)u](n) := V(n_1,...,n_i+1,...n_d)u(n)\right).
\end{equation*} 
We will also be using the bracket notation $\langle \alpha \rangle := \sqrt{1+|\alpha|^2}$. Let us now get into the details of the potential $V$. All in all, we will require two conditions on $V$: we suppose that there exist $\rho, C >0$ such that 
\begin{align}
\label{kdkd5}
\langle n \rangle ^{\rho} |V(n)| &\leqslant C, \quad \text{for all} \ n \in \Z^d, \quad \text{and} \\
\label{kdkd2}
\langle n \rangle ^{\rho} |n_i| |(V - \tau_iV)(n)| &\leqslant C, \quad \text{for all} \ n \in \Z^d \ \text{and} \ 1 \leqslant i \leqslant d.
\end{align}
These conditions can be interpreted as a discrete version of the standard long-range type potential $|\partial^{\alpha}V(x)| \leqslant C \langle x\rangle  ^{-|\alpha| - \rho}$ in the continuous case. Examples of potentials $V$ satisfying these two conditions include $|V(n)| \leqslant C \langle n \rangle ^{-1-\rho}$, the so-called short-range potential, and $V(n) = C \langle n \rangle ^{-\rho}$.

The goal of this paper is to establish the LAP for $H$ as defined in \eqref{low1}. The formulation of the LAP requires a conjugate operator which we now introduce. But first, we need the position operator $N := (N_1,...,N_d)$, where the $N_i$ are defined by
\begin{equation*}
(N_iu)(n) := n_i u(n), \quad \mathcal{D}(N_i) = \Big \{ u \in \mathcal{H} : \sum_{n \in \Z^d} |n_i u(n)|^2 < \infty \Big \}.
\end{equation*}
The conjugate operator to $H$ will be the generator of dilations denoted $A$ and is the closure of
\begin{equation}
\label{generatorDilations}
A_0 := \i\sum_{i=1}^d  \left(2^{-1}(S_i^*+S_i) - (S_i^*-S_i)N_i\right) = \i  \sum_{i=1}^d 2^{-1} \big( (S_i-S_i^*)N_i + N_i(S_i-S_i^*) \big)
\end{equation} 
defined on $\mathcal{D}(A_0) = \ell_0(\Z^d)$, the collection of sequences with compact support. The operator $A$ is self-adjoint. We will also make use of the projectors onto the pure point spectral subspace of $H$ and its complement, denoted $P$ and $P^{\perp} := 1-P$ respectively. We define the following sets:
\begin{align}
\mu(H) = \mu(H') &:= (0,4)\setminus \{E_{\pm}(k)\} \ \ \text{for} \ \ d=1, \\
\mu(H) &:= [0,E(k)) \cup (4d-E(k),4d] \ \ \text{for} \ \ d\geqslant 2, \\
\mu(H') &:= [0,E'(k))\cup(4d-E'(k),4d] \ \ \text{for} \ \ d\geqslant 2.
\end{align}
Recall $E_{\pm}(k)$ defined by \eqref{yummy}. The definitions of $E(k)$  and $E'(k)$ are respectively given in Propositions \ref{sj9} and \ref{theProp33}. We may as well already mention that the sets $\mu$ consist of points where the classical Mourre estimate holds for $H$ and $H'$. The main result of the paper is the following:
\begin{theorem}
\label{lapy}
Let $E \in \mu(H)$. Then there is an open interval $\mathcal{I}$ containing $E$ such that $H$ has finitely many eigenvalues in $\mathcal{I}$ and these are of finite multiplicity. Furthermore, if $\ker (H-E) \subset D(A)$, then $\mathcal{I}$ can be chosen so that for any $s > 1/2$ and any compact interval $\mathcal{I}' \subset \mathcal{I}$, the reduced LAP for $H$ holds with respect to $(\mathcal{I}',s,A)$, that is to say,
\begin{equation}
\label{RedLAP}
\sup \limits_{x \in \mathcal{I}', y \neq 0} \|\langle A \rangle ^{-s} (H-x-\i y)^{-1}P^{\perp} \langle A \rangle ^{-s} \| < \infty.
\end{equation}
In particular, the following local decay estimate holds:
\begin{equation}
\int _{\R} \| \langle N \rangle ^{s} e^{-\i tH} P^{\perp} \theta (H) u \|^2 dt \leqslant C \|u\|^2, \quad \text{for any} \ u \in \mathcal{H}, \theta \in C_c^{\infty}(\R), \ \text{and} \ s>1/2,
\end{equation}
and the spectrum of $H$ is purely absolutely continuous on $\mathcal{I}'$ whenever $P =0$ on $\mathcal{I}'$.
\end{theorem}
The corresponding result also holds for $H'$. The last part of Theorem \ref{lapy} are two well-known consequences of the LAP. The local decay estimate gives a better insight into how the initial state $\theta(H)u$ diverges to infinity.

Our result is a discrete version of the LAP for the corresponding continuous  Schr{\"o}dinger operator obtained by Gol{\'e}nia and Jecko in \cite{GJ2}, and our proof is very much inspired from theirs. The proof is based on variations of classical Mourre theory. Classical Mourre theory was proven very successful to study the point and continuous spectra of a wide class of self-adjoint operators. A standard reference is the book \cite{ABG} in which optimal results are obtained for a wide class of potentials, and we also refer to \cite{Sa}.

In \cite{GJ1} and \cite{GJ2}, a new approach to Mourre's theory is developed. Their approach proves the LAP without the use of differential inequalities. In the separate work of G{\'e}rard \cite{G}, he proves the LAP using traditional energy estimates and introduces weighted Mourre estimates. In \cite{GJ2}, Gol{\'e}nia and Jecko are able to prove the LAP under weaker conditions on the potential than what is usually assumed in \cite{ABG} or \cite{Sa} for example, because their starting point is not the classical Mourre estimate but rather the weighted Mourre estimate. Roughly speaking, the original Mourre theory required $[[V,A],A]$ to be bounded in a weak sense, whereas the more recent and different approaches require $V$ to belong to a class where solely $[V,A]$ is bounded. This allows for new classes of potentials to be studied, such as the Wigner-von Neumann potential. In Propositions \ref{NotC1u} and \ref{notC1u2}, we show that the standard Mourre commutator techniques exposed in \cite{ABG} or \cite{Sa} cannot be used to study the Wigner-von Neumann potential. Finally, the LAP derived in this paper is interesting because we include a long-range type potential $V$ in addition to the Wigner-von Neumann potential and therefore provide new results including the question of the absolutely continuous spectrum.

The paper is organized as follows: In Section \ref{preliminaries}, we recall the basic notions of classical and weighted Mourre theory that we will be using. In the Section \ref{1d}, we study the classical Mourre theory for the one-dimensional Schr\"{o}dinger operators $H$ and $H'$, and show that the Wigner-von Neumann potential cannot be treated with the classical methods. In Section \ref{ddim}, we repeat our analysis for the multi-dimensional Schr\"{o}dinger operators $H$ and $H'$, and recycle results from the one-dimensional case. In Section \ref{WMT}, we prove the weighted Mourre estimate that leads to the LAP. This section is done independently of the dimension. Finally in the Appendix \ref{appendix}, we recall essential facts about almost analytic extensions of $C^{\infty}(\R)$ functions that we need to establish the weighted Mourre estimate. 
\vspace{0.5cm}

\noindent \textbf{Acknowledgments:} It is a pleasure to thank my thesis supervisor Sylvain Gol\'{e}nia for offering me this topic and his generous support and guidance throughout my research. I would like to thank the university of Bordeaux for funding my studies.


\section{Preliminaries}
\label{preliminaries}
\subsection{Regularity}
We consider two self-adjoint operators $T$ and $A$ acting in some complex Hilbert space $\mathcal{H}$, and for the purpose of the overview $T$ will be bounded. Given $k \in \N$, we say that $T$ is of class $\mathcal{C}^k(A)$, and write $H \in \mathcal{C}^k(A)$ if the map  
\begin{equation}
\label{hey88}
\R \ni t \mapsto e^{\i tA}Te^{-\i tA} u \in \mathcal{H}
\end{equation} 
has the usual $C^k(\R)$ regularity for every $u \in \mathcal{H}$. Let $\mathcal{I}$ be an open interval of $\R$. We say that $T$ is locally of class $\mathcal{C}^k(A)$ on $\mathcal{I}$, and write $T \in \mathcal{C}^k_{\mathcal{I}}(A)$, if for all $\varphi \in C^{\infty}_c(\R)$ with support in $\mathcal{I}$, $\varphi(T) \in \mathcal{C}^k(A)$. 
The form $[T,A]$ is defined on $\mathcal{D}(A) \times\mathcal{D}(A)$ by 
\begin{equation}
\label{form}
\langle u , [T,A] v \rangle := \langle Tu,Av \rangle - \langle Au , Tv \rangle.
\end{equation}
We recall the following result of \cite[p. 250]{ABG}:
\begin{proposition} Let $T \in \mathcal{B}(\mathcal{H})$, the bounded operators on $\mathcal{H}$. The following are equivalent:
\begin{enumerate}
\item $T \in \mathcal{C}^1(A)$.
\item The form $[T,A]$ extends to a bounded form on $\mathcal{H}\times \mathcal{H}$ defining a bounded operator denoted by $\text{ad}^1_A(T) := [T,A]_{\circ}$.
\item $T$ preserves $\mathcal{D}(A)$ and the operator $TA-AT$, defined on $\mathcal{D}(A)$, extends to a bounded operator.
\end{enumerate}
\end{proposition}
Consequently, $T \in \mathcal{C}^k(A)$ if and only if the iterated commutators $\text{ad}^p_A(T) := [\text{ad}^{p-1}_A(T), A]_{\circ}$ are bounded for $1 \leqslant p \leqslant k$.
We recall a general Lemma which can found in \cite[section 2]{GGM}:
\begin{Lemma}
\label{lembem}
The class $\mathcal{C}^1(A)$ is a $*$-algebra, that is, for $T_1,T_2 \in \mathcal{C}^1(A)$ we have:
\begin{enumerate}
\item $T_1+T_2 \in \mathcal{C}^1(A)$ and $[T_1+T_2,A]_{\circ} = [T_1,A]_{\circ} + [T_2,A]_{\circ}$. 
\item $T_1 T_2 \in \mathcal{C}^1(A)$ and  $[T_1 T_2,A]_{\circ} = T_1[T_2,A]_{\circ} + [T_1,A]_{\circ}T_2$.
\item $T_1^* \in \mathcal{C}^1(A)$ and $[T_1^*,A]_{\circ} = [T_1,A]_{\circ}^*$.
\end{enumerate}
\end{Lemma}
Finally we will also need the following result from \cite{GJ1}:
\begin{proposition} For $u,v \in \mathcal{D}(A)$, the rank one operator $|u \rangle \langle v |: w \to \langle v, w \rangle u$ is of class $\mathcal{C}^1(A)$.
\label{rank one}
\end{proposition}


\subsection{The scale of the different classes}

Let us introduce other classes inside $\mathcal{C}^1(A)$. We say that $T \in \mathcal{C}^{1,\text{u}}(A)$ if the map 
\begin{equation}
\R \ni t \mapsto e^{\i tA} T e^{-\i tA} \in \mathcal{B}(\mathcal{H})
\end{equation} 
has the usual $C^1(\R)$ regularity. Note the difference with definition \eqref{hey88}. We say that $T \in \mathcal{C}^{1,1}(A)$ if 
\begin{equation}
\int_0 ^1 \| [T, e^{\i tA} ]_{\circ} , e^{\i tA}]_{\circ} \| t^{-2} dt < \infty.
\end{equation}
Finally we say that $T \in \mathcal{C}^{1+0}(A)$ if $T \in \mathcal{C}^{1}(A)$ and 
\begin{equation}
\int_{-1} ^1 \| e^{\i tA} [T, A]_{\circ} e^{-\i tA} \| |t|^{-1} dt < \infty.
\end{equation}
It turns out that 
\begin{equation}
\label{scaleClasses}
\mathcal{C}^2(A) \subset \mathcal{C}^{1+0}(A) \subset \mathcal{C}^{1,1}(A) \subset \mathcal{C}^{1,\text{u}}(A) \subset \mathcal{C}^1(A).
\end{equation}
The local classes are defined in the obvious way: $T \in \mathcal{C}_{\mathcal{I}}^{[\cdot]}(A)$ if, for all $\varphi \in C_c^{\infty}(\mathcal{I})$, $\varphi(T) \in \mathcal{C}^{[\cdot]}(A)$.

In \cite{Sa}, the LAP is obtained on compact subintervals of $\mathcal{I}$ when $H \in \mathcal{C}_{\mathcal{I}}^{1+0}(A)$, while in \cite[section 7.B]{ABG}, it is obtained for $H \in \mathcal{C}^{1,1}(A)$ and this class is shown to be optimal among the global classes in the framework. 


\subsection{The Mourre estimate and the LAP}
Let $\mathcal{I},\mathcal{J}$ be open intervals with $\overline{\mathcal{I}} \subset \mathcal{J}$, and assume $T \in \mathcal{C}^1 _{\mathcal{J}}(A)$. We say that the \textit{Mourre estimate} holds for $T$ on $\mathcal{I}$ if there exist a finite $c > 0$ and a compact operator $K$ such that 
\begin{equation}
E_{\mathcal{I}}(T) [T,\i A ]_{\circ} E_{\mathcal{I}}(T)  \geqslant c \cdot E_{\mathcal{I}}(T) +K
\label{Mourre K}
\end{equation}
in the form sense on $\mathcal{D}(A) \times \mathcal{D}(A)$. 
We say that the \textit{strict Mourre estimate} holds for $T$ on $\mathcal{I}$ if \eqref{Mourre K} holds with $K=0$. 
Assuming $\overline{\mathcal{I}} \subset \mathcal{J}$ and $T \in \mathcal{C}^1_{\mathcal{J}}(A)$, there are finitely many eigenvalues of $T$ in $\mathcal{I}$ and they are of finite multiplicity when $K \neq 0$; whereas $T$ has no eigenvalues in $\mathcal{I}$ when $K=0$. This is a direct and easy consequence of the Virial Theorem (\cite{Sa}, \cite[Proposition 7.2.10]{ABG}). Recent variations of classical Mourre theory make use of a \textit{weighted Mourre estimate} : 
\begin{equation}
\label{weighted Mourre}
E_{\mathcal{I}}(T) [T,\i \varphi(A) ]_{\circ} E_{\mathcal{I}}(T)  \geqslant E_{\mathcal{I}}(T) \langle A \rangle ^{-s}(c+K) \langle A \rangle ^{-s} E_{\mathcal{I}}(T)
\end{equation}
where $0<c<\infty$, $s>1/2$ and $\varphi$ is some function in $B_b(\R)$, the bounded Borel functions. This type of estimate appears in \cite{G} and \cite{GJ2}. Recall that $P$ is the orthogonal projection onto the pure point spectral subspace of $H$, and $P^{\perp} := 1 -P$. We now quote the essential criterion established in \cite{GJ2} that we will need to prove the LAP for $H$ as defined in \eqref{low1}.

\begin{theorem}\cite{GJ2} Let $\mathcal{I}$ be an open interval, and assume that $P^{\perp}\theta(T) \in \mathcal{C}^1(A)$ for all $\theta \in C^{\infty}_c(\mathcal{I})$. Assume the existence of an $s_0 \in (1/2,1]$ with the following property : for any $s \in (1/2,s_0]$, there exist a finite $c>0$ and a compact operator $K$ such that for all $R \geqslant 1$, there exists $\psi_R \in B_b(\R)$ so that the following projected weighted Mourre estimate
\begin{equation}
\label{projected weighted Mourre}
P^{\perp} E_{\mathcal{I}}(T) [T,\i \varphi_R(A/R) ]_{\circ} E_{\mathcal{I}}(T) P^{\perp}  \geqslant P^{\perp} E_{\mathcal{I}}(T) \langle A/R \rangle ^{-s} (c+K) \langle A/R \rangle ^{-s} E_{\mathcal{I}}(T) P^{\perp}
\end{equation}
holds. Then for all $s>1/2$ and compact $\mathcal{I}'$ with $\overline{\mathcal{I}'} \subset \mathcal{I}$, the reduced LAP \eqref{RedLAP} for $T$ holds with respect to $(\mathcal{I}',s,A)$.
\label{hoyaya}
\end{theorem}



\section{The One-Dimensional Case}
\label{1d}
 
We begin with the study of the one-dimensional operator. We write the Laplacian in terms of the shift operators defined in \eqref{ShiftOperator} : $\Delta = 2 -(S^*+S)$. Note that $[S, \Delta]_{\circ} = [S^*, \Delta]_{\circ} = 0$. Recall that $A$ is the conjugate operator to $H$ introduced in \eqref{generatorDilations}. It is the closure of the operator:
\begin{equation}
\label{do34}
A_0 := -\i \left( 2^{-1}(S^* + S)+N(S^*-S) \right) = \i \left( 2^{-1}(S^* + S) - (S^*-S)N \right), \quad \mathcal{D}(A_0) = \ell_0(\Z).
\end{equation}
The domain of $A$ has been explicitely shown to be $\mathcal{D}(A) = \mathcal{D}(N(S^*-S))$ and this operator has been shown to be self-adjoint (cf.\ \cite{GGo}). Moreover $A$ is unitarily equivalent to the self-adjoint realization of the operator
\begin{equation*}
A_{\mathcal{F}} : =  \i\sin(\xi) \frac{d}{d\xi} + \i\frac{d}{d\xi} \sin(\xi), \quad  \mathcal{D}(A_{\mathcal{F}}) := \{ f \in L^2([-\pi, \pi], d\xi) : A_{\mathcal{F}}f \in L^2([-\pi, \pi], d\xi) \}.
\end{equation*}

\subsection{\texorpdfstring{$\mathcal{C}^1(A)$}{CA} Regularity}
We now show that $H$ and $H'$ are of class $\mathcal{C}^1(A)$. 

\begin{proposition}
\label{oieoie1}
The form $[\Delta,\i A]$ extends to a bounded form denoted $[\Delta,\i A]_{\circ}$ and
\begin{equation}
\label{oieoie}
[\Delta, \i A] _{\circ}= \Delta(4-\Delta).
\end{equation}
Furthermore $\Delta$ is of class $\mathcal{C}^{\infty}(A)$. 
\end{proposition}

\proof
A straightforward and well-known computation shows that $\langle u,[\Delta, \i A]v\rangle = \langle u,\Delta(4-\Delta) v \rangle $ for all $u,v \in \ell_0(\Z)$. Thus $[\Delta, 
\i A]$ extends to a bounded form and we have \eqref{oieoie}. Using induction and applying Lemma \ref{lembem} shows that $\text{ad}_A^{k}(\Delta)$ is a polynomial of degree $k+1$ in $\Delta$.
\qed  

For future reference when we deal with the multi-dimensional case, let us introduce the following function $\varrho_{T}^A :\R \to (-\infty,+\infty]$ defined for a pair of self-adjoint operators $T$ and $A$: 
\begin{equation*}
\varrho_{T}^A(E) := \sup \ \big\{ a\in \R: \exists \epsilon>0 \  \text{such that} \  E_{\mathcal{I}(E;\epsilon)}(T)[T,\i A]_{\circ} E_{\mathcal{I}(E;\epsilon)}(T) \geqslant a \cdot E_{\mathcal{I}(E;\epsilon)}(T) \big\}.
\label{VARRHO}
\end{equation*}
Here $\mathcal{I}(E;\epsilon)$ is the open interval of radius $\epsilon>0$ centered at $E \in \R$. It is known for example that $\varrho_{T}^A$ is lower semicontinuous and $\varrho_{T}^A(E) < \infty$ if and only if $E \in \sigma(T)$. For more properties of this function, see \cite[chapter 7]{ABG}. As a consequence of \eqref{oieoie}, we have for $E \in (0,4)$:
\begin{equation}
\varrho_{\Delta}^A(E) = E(4-E).
\label{varrhoDelta}
\end{equation}
Define the bounded operators 
\begin{equation}
K_{W} := 2^{-1}W(S^*+S) + 2^{-1}(S^*+S)W, \quad \text{and} \quad B_W:= U\tilde{W}(S^*-S)-(S^*-S)\tilde{W}U,    
\end{equation}
where $\tilde{W}$ is the operator $(\tilde{W}u)(n) := q\sin(kn) u(n)$ and $U$ is the operator \\
$(U u)(n) := \delta_{\Z\setminus \{0\}}(n)\sign(n)u(n)$. A simple calculation shows that for all $u,v\in \ell_0(\Z)$, 
\begin{equation}
    \langle u,[W,\i A] v \rangle = \langle u, K_{W} v \rangle + \langle u, B_{W} v \rangle.
\end{equation}
We also investigate the form $[W',\i A]$. Define the bounded operators 
\begin{equation}
\label{easyu7}
K_{W'} := 2^{-1}W'(S^*+S) + 2^{-1}(S^*+S)W', \quad \text{and} \quad B_{W'} := \tilde{W}(S^*-S) -(S^*-S)\tilde{W}.
\end{equation}
A straightforward computation shows that for all $u,v \in \ell_0(\Z)$,
\begin{equation*}
\langle u,[W',\i A]v \rangle = \langle u,K_{W'} v \rangle + \langle u,B_{W'} v \rangle. 
\end{equation*}
Hence both $[W,\i A]$ and $[W',\i A]$ extend to bounded forms and we have
\begin{equation}
\label{lapd1}
[W,\i A]_{\circ} = K_{W} + B_{W}, \qquad \text{and} \ [W',\i A]_{\circ} = K_{W'} + B_{W'}.
\end{equation}
Note that $K_W$ and $K_{W'}$ are compact, while $B_W$ and $B_{W'}$ are bounded (but not compact by Proposition \ref{NotC1u}). Finally, we turn to the form $[V,\i A]$. For $u \in \ell_0(\Z)$ we have 
\begin{equation*}
\label{cocod}
((VA-AV)u)(n) = \i (n-2^{-1})(V(n)-V(n-1))u(n-1) + \i (n-2^{-1})(V(n)-V(n+1))u(n+1).
\end{equation*}
Therefore in the form sense, we have for $u,v \in \ell_0(\Z)$, 
\begin{equation}
\langle u,[V,\i A]v \rangle = -\langle u, \big[(N-2^{-1})(V-\tau V)S+(N-2^{-1})(V-\tau ^*V)S^*\big] v \rangle.
\label{Goddy}
\end{equation} 
By hypothesis \eqref{kdkd2}, we see that $[V,\i A]$ can be extended to a bounded form, and that $[V,\i A]_{\circ}$ is a compact operator. The above discussion leads to: 
\begin{proposition}
$H = \Delta + W + V$ and $H' = \Delta + W' + V$ are of class $\mathcal{C}^1(A)$. 
\end{proposition}
We now explain why the usual Mourre theory with conjugate operator $A$ cannot be applied. We have proved that $H \in \mathcal{C}^1(A)$, however in order to apply the standard Mourre theory, one typically has to prove that $H$ is in a better class of regularity w.r.t. $A$. As mentionned previously, the existing standard theory in \cite{ABG} is optimal for the class $\mathcal{C}^{1,1}(A)$. However, we are not dealing with potentials in this class as shown in the following Proposition. The same phenomenon occurs in the case of the continuous Schr\"odinger operator (cf.\ \cite{GJ2}).

\begin{proposition}
$H$ and $H'$ are not of class $\mathcal{C}^{1,\text{u}}(A)$.
\label{NotC1u}
\end{proposition}
\proof
We stick with $H$ as the same proof works for $H'$. Since $\Delta \in \mathcal{C}^{\infty}(A)$, we have $\Delta \in \mathcal{C}^{1,\text{u}}(A)$. Let us assume by contradiction that $H \in \mathcal{C}^{1,\text{u}}(A)$. Then $H-\Delta \in \mathcal{C}^{1,\text{u}}(A)$. In particular 
\begin{equation*}
\lim \limits_{t \to 0} \big[e^{-\i tA}(H-\Delta)e^{\i tA} - (H-\Delta)\big]t^{-1} = [(H-\Delta), \i A]_{\circ} = [(W+V), \i A]_{\circ}
\end{equation*}
is a compact operator as the norm limit of compact operators. As explained before, $[V, \i A]_{\circ}$ is compact, and $[W, \i A]_{\circ}$ is the sum of the compact operator $K_W$ and the bounded operator $B_W$. We show that $B_W$ is not compact, and this will be our contradiction. Consider the sequence $(\delta_j)_{j \geqslant 2}$ of unit vectors in $\ell^2(\Z)$ satisfying $(\delta_j)(n) = \delta_{j;n}$. Then 
\begin{equation*}
B_W \delta_j = q \left(\sin(k(j-1))-\sin(kj)\right) \delta_{j-1} - q \left(\sin(k(j+1))-\sin(kj)\right) \delta_{j+1}.
\end{equation*}
For this operator to be compact, we require 
\begin{align*}
0 &= \lim \limits_{j \to \infty} |q||\sin(k(j-1))-\sin(kj)| + |q||\sin(k(j+1))-\sin(kj)| \\
&= \lim \limits_{j \to \infty} 2|q||\cos(kj-k/2)||\sin(k/2)| + 2|q||\cos(kj+k/2)||\sin(k/2)|.
\end{align*}
As $j\to \infty$, we would need $kj-k/2 \to \pi/2$ (mod $\pi$) and $kj+k/2 \to \pi/2$ (mod $\pi$), but this is not possible precisely because $k \not \in \pi \Z$. 
\qed



\subsection{Classical Mourre Theory}
In this section we derive the classical Mourre estimate \eqref{Mourre K} for the one-dimensional Schr{\"o}dinger operator $H$. From the previous section, we know that $[V, \i A]_{\circ}$ is compact and that $[W, \i A]_{\circ} = K_W +B_W$, with $K_W$ compact but $B_W$ just bounded. Therefore, in order to derive the Mourre estimate, what really remains to show is that $E_{\mathcal{I}}(\Delta) B_W E_{\mathcal{I}}(\Delta)$ is compact for some well-chosen $\mathcal{I} \subset (0,4)$. We show precisely: 

\begin{Lemma} 
\label{ImportantLemma}
Recall that $E_{\pm}(k) := 2 \pm 2\cos (k/2)$. Let $E \in [0,4]\setminus\{E_{\pm}(k) \}$. Then there exists $\epsilon(E) >0$ such that for all $\theta \in C_c^{\infty}(\R)$ supported on $\mathcal{I}:=(E-\epsilon, E+\epsilon)$, $\theta(\Delta)\tilde{W}\theta(\Delta)=0$. Thus $\theta(\Delta) B_{W'} \theta(\Delta)=0$ and $\theta(\Delta) B_{W} \theta(\Delta)$ is compact.
\end{Lemma}
The proof of this Lemma is deferred to the end of this section, but note that the last part of the Lemma is easy, since if $\theta(\Delta)\tilde{W}\theta(\Delta)=0$, then
\begin{equation*}
\theta(\Delta) B_{W'} \theta(\Delta) = \theta(\Delta) \tilde{W} \theta(\Delta)(S^*-S)-(S^*-S)\theta(\Delta) \tilde{W} \theta(\Delta) = 0.
\end{equation*}
Commuting $U$ with $\Delta$ produces a finite rank, hence compact operator by \eqref{swift5}, so $\theta(\Delta) B_{W} \theta(\Delta)$ is compact. The classical Mourre estimate for $H$ and $H'$ is easily deduced:
\begin{proposition}
For every $E \in (0,4) \setminus \{E_{\pm}(k)\}$, there is an open interval $\mathcal{I}$ containing $E$ such that the Mourre estimate \eqref{Mourre K} holds for $H$ and $H'$. In particular, the number of eigenvalues of $H$ and $H'$ in $\mathcal{I}$ are finite and they are of finite multiplicity.
\label{goosh}
\end{proposition}
\proof
For $E \in (0,4) \setminus \{E_{\pm}(k)\}$, let $\theta \in C_c^{\infty}(\R)$ be as in Lemma \ref{ImportantLemma}, with supp$(\theta) = \mathcal{I}$. By the resolvent identity, $\Omega := \theta(H) - \theta(\Delta)$ is compact. We have for some operator $K$:
\begin{align*}
\theta(H) [H, \i A ]_{\circ} \theta(H) &= \theta(\Delta) [H, \i A ]_{\circ} \theta(\Delta) + \Omega [H, \i A ]_{\circ} \theta(H) + \theta(\Delta) [H, \i A ]_{\circ} \Omega \\
&= \theta(\Delta) \Delta (4-\Delta) \theta(\Delta) + K.
\end{align*}
By functional calculus, $\theta(\Delta) \Delta (4-\Delta) \theta(\Delta) \geqslant c \theta^2(\Delta)$, $c := \min_{x \in \mathcal{I}} x(4-x)$. Thus
\begin{equation*}
\label{suy7}
\theta(H) [H, \i A ]_{\circ} \theta(H) \geqslant c \theta^2(H) + K + c (\theta^2(\Delta)-\theta^2(H)).  
\end{equation*}
For all open intervals $\mathcal{I}'$ with $\overline{\mathcal{I}'} \subset \mathcal{I}$, we obtain the Mourre estimate when applying $E_{\mathcal{I}'}(H)$ to either sides of the last inequality.
\qed

We now show that compactness of $E_{\mathcal{I}}(H) B_W E_{\mathcal{I}}(H)$ is not possible for all intervals $\mathcal{I}$ centered about $E_{\pm}(k)$. Thanks to the relations
\begin{equation}
\label{swift5}
    S^*U = US^* + \delta_{\{0\}}S^* +\delta_{\{-1\}}S^*, \quad \text{and} \quad S U = U S - \delta_{\{0\}}S - \delta_{\{1\}}S,
\end{equation}
one shows that
\begin{equation}
\label{johny}
\theta(\Delta)B_W\theta(\Delta) = U\theta(\Delta)  \left( \tilde{W}(S^*-S)-(S^*-S)\tilde{W}\right) \theta(\Delta) + \ \text{compact}.  
\end{equation}

\begin{proposition}
Fix $E \in \{E_{\pm}(k)\}$. For all $\theta \in C_c^{\infty}(R)$ with supp$(\theta) \ni E$, $\theta(\Delta) B_W \theta(\Delta)$ and $\theta(\Delta) B_{W'} \theta(\Delta)$ are not compact. 
\end{proposition}
\begin{proof}
We show that $Q := \theta(\Delta)  [ \tilde{W}(S^*-S)-(S^*-S)\tilde{W} ] \theta(\Delta) = \theta(\Delta) B_{W'} \theta(\Delta)$ is not compact for all $\theta$ supported about $E_{\pm}(k)$. Applying $U$ to this operator does not make the product any more compact, and so the result will follow by \eqref{johny}. In Fourier space, $Q$ becomes \begin{equation}
\theta(2-2\cos(\cdot))\big([q(T_k-T_{-k})/2\i][-2\i\sin(\cdot)] -[-2\i\sin(\cdot)][q(T_k-T_{-k})/2\i]\big)\theta(2-2\cos(\cdot)). 
\end{equation}
Here $T_k$ is the operator of translation by $k$. It is not hard to see that if $\phi$ solves
\begin{equation}
    2-2\cos(\phi)=2-2\cos(\phi+k), \quad \text{or} \quad 2-2\cos(\phi)=2-2\cos(\phi-k)
\end{equation}
then it is possible to construct a sequence of \guillemotleft delta\guillemotright \ functions $f_n$ supported in a neighborhood of $\phi$ converging weakly to zero, but $\|f_n\|_2 =1$. The solutions to the previous equations are $\phi=k/2, \pi-k/2$ for the first, and $\phi=-k/2,\pi+k/2$ for the second. Either way, we retrieve the threshold energies $E_{\pm}(k) = 2\pm 2\cos(k/2)$.
\qed
\end{proof}




The rest of the section is devoted to proving Lemma \ref{ImportantLemma}. For future reference for the multi-dimensional case, what follows is done in $d$ dimensions. Let $k \in \mathcal{T} := \R \setminus \pi \Z \ (\text{mod} \ 2\pi)$, and let $T_k$ be the multiplication operator on $\ell^2(\Z^d)$ given by $(T_k u)(n) := e^{\i k(n_1+...+n_d)}u(n)$. Then $T_k$ corresponds to a translation in the Fourier space of $2\pi$-periodic functions by $k$ in each direction, that is, $(\mathcal{F}T_{k}\mathcal{F}^{-1}f)(\xi)=f(\xi+k)$ (see \eqref{shammy} for the definition of the Fourier transform). We analyze how $\Delta$ and $T_k$ commute. We have: 
\begin{align}
(\mathcal{F}T_k\Delta\mathcal{F}^{-1}f)(\xi) &= f(\xi+k) \sum_{i=1}^d(2-2\cos(\xi_i+k)) \nonumber \\
&= 
\label{1123}
f(\xi+k)\sum_{i=1}^d \big[ 2-2\cos (k) \cos (\xi_i) -2\sin(k)\sqrt{1-\cos^2(\xi_i)}(2 \1_{[0,\pi]}(\xi_i)-1) \big].
\end{align}
Denote by $\widecheck{\1}_{[0,\pi],i}$ the operator on $\ell^2(\Z^d)$ satisfying $(\mathcal{F}\widecheck{\1}_{[0,\pi],i}\mathcal{F}^{-1}f)(\xi) = \1_{[0,\pi]}(\xi_i) f(\xi)$. Note that $\widecheck{\1}_{[0,\pi],i}$ is a bounded self-adjoint operator with spectrum $\sigma(\widecheck{\1}_{[0,\pi],i}) = \text{ess ran }(\1_{[0,\pi]}(\xi_i)) = \{0,1\}$. Moreover $\widecheck{\1}_{[0,\pi],i}$ commutes with $\widecheck{\1}_{[0,\pi],j}$ and $\Delta_j$ for all $1 \leqslant i,j \leqslant d$. Here $\Delta_j$ is the Laplacian restricted to the $j^{\text{th}}$ dimension : $(\mathcal{F}\Delta_j\mathcal{F}^{-1}f)(\xi) = f(\xi)(2-2\cos(\xi_j))$ and $\sigma(\Delta_j) = [0,4]$. For fixed $k \in \mathcal{T}$, define the function $g_k$ for $(x,y) \in [0,4] \times \{0,1\}$ by 
\begin{equation}
g_k(x,y) := 2+(x-2)\cos(k) - \sin(k)\sqrt{x(4-x)}(2y-1)
\label{okid}
\end{equation}
and the bounded self-adjoint operators on $\ell^2(\Z^d)$,
\begin{equation*}
g_k(\Delta_i, \widecheck{\1}_{[0,\pi],i}) := 2+(\Delta_i-2)\cos(k) -\sin(k)\sqrt{\Delta_i(4-\Delta_i)}(2 \widecheck{\1}_{[0,\pi],i}-1).
\end{equation*}
Then \eqref{1123} provides us with the following key relation: 
\begin{equation*}
T_k \Delta = \left(\sum_{i=1}^d g_k(\Delta_i, \widecheck{\1}_{[0,\pi],i}) \right)T_k.
\end{equation*}
In particular, for all $z \in \C \setminus \R$, 
\begin{equation*}
\label{vd=bv}
T_k (z-\Delta)^{-1} = \left(z-\sum_{i=1}^d g_k(\Delta_i, \widecheck{\1}_{[0,\pi],i})\right)^{-1}T_k.
\end{equation*}
For $\theta \in C_c^{\infty}(\R)$, the Helffer-Sj{\"o}strand formula yields 
\begin{equation*}
\theta \left(\sum_{i=1}^d g_k(\Delta_i, \widecheck{\1}_{[0,\pi],i})\right) = \frac{\i}{2\pi} \int_{\C} \frac{\partial \tilde{\theta}}{\partial \overline{z}} \left(z-\sum_{i=1}^d g_k(\Delta_i, \widecheck{\1}_{[0,\pi],i})\right)^{-1} \dz.
\end{equation*} 
Hence we have derived the following formula :
\begin{equation}
\label{vd=bv54}
T_k \theta(\Delta) = \theta \left(\sum_{i=1}^d g_k(\Delta_i, \widecheck{\1}_{[0,\pi],i}) \right)T_k.
\end{equation}
Since $\{\Delta_i,\widecheck{\1}_{[0,\pi],i}\}_{i=1}^d$ forms a family of $2d$ self-adjoint commuting operators, we may apply the functional calculus for such operators. We are now ready to prove Lemma \ref{ImportantLemma}.


\noindent \textit{Proof of Lemma \ref{ImportantLemma}}.
We apply \eqref{vd=bv54} with $d=1$ and get  
\begin{equation*}
\theta(\Delta)\tilde{W}\theta(\Delta) = \theta(\Delta)\theta(g_k(\Delta, \widecheck{\1}_{[0,\pi]}))qT_k/(2\i) - \theta(\Delta)\theta(g_{2\pi-k}(\Delta, \widecheck{\1}_{[0,\pi]}))qT_{2\pi-k}/(2\i).
\end{equation*} 
We show that for all $k \in \mathcal{T}$, one may choose $\theta$ appropriately so that $\theta(\Delta)\theta(g_k(\Delta, \widecheck{\1}_{[0,\pi]})) = 0$. Also, as will be seen shortly, we will have $\theta(\Delta)\theta(g_k(\Delta, \widecheck{\1}_{[0,\pi]})) = 0$ iff $\theta(\Delta)\theta(g_{2\pi-k}(\Delta, \widecheck{\1}_{[0,\pi]})) = 0$. We appeal to the functional calculus for two self-adjoint commuting operators. Consider the function $g_k(x,y)$ of \eqref{okid} defined for $(x,y) \in \sigma(\Delta)\times\sigma(\widecheck{\1}_{[0,\pi]}) = [0,4]\times\{0,1\}$. We show that for all $E \in [0,4] \setminus \{E_{\pm}(k) \}$, there exists $\epsilon(E) > 0$ such that for the interval $\mathcal{I} := (E-\epsilon,E+\epsilon)$,
\begin{equation}
\mathcal{I} \cap \{ g_k(x,y) : x \in \mathcal{I}, y \in \{0,1\} \} = \emptyset.
\label{ok}
\end{equation} 
In this way if supp$(\theta)=\mathcal{I}$, then we will have $\theta(x)\theta(g_k(x,y)) = 0$ as required. Set 
\begin{equation}
\mathcal{E}(k) := \{ E \in [0,4] : \ \text{there exists} \ y \in \{0,1\} \ \text{such that} \ E = g_k(E,y) \}.   
\end{equation}
Clearly if $E \in \mathcal{E}(k)$, then \eqref{ok} does not hold at $E$. To simplify the analysis, we let 
\begin{equation}
g_{k;\pm}(x) := 2+(x-2)\cos(k) \pm \sin(k)\sqrt{x(4-x)} \quad  \text{and} \quad h_{k;\pm}(x) := g_{k;\pm}(x)-x.
\label{vg}
\end{equation}
Notice that $h_{k;\pm} (E_{\pm}(k))=0$, and so $E_{\pm}(k) \in \mathcal{E}(k)$. To show that $\mathcal{E}(k) = \{E_-(k),E_+(k)\}$, it is equivalent to show that these are the only roots of $h_{k;\pm}$. Because of the symmetry relations 
\begin{equation}
g_{k;+}(x) = 4-g_{k;-}(4-x)  \quad \text{and} \quad h_{k;+}(x) = -h_{k;-}(4-x),
\label{symmetry}
\end{equation}



\begin{figure}
\begin{tikzpicture}[domain=0.01:5]
\begin{axis}
[grid = major, 
axis x line = middle, 
axis y line = middle, 
xlabel={$x$}, 
xlabel style={at=(current axis.right of origin), anchor=west}, 
ylabel={$y$}, 
ylabel style={at=(current axis.above origin), anchor=south}, 
domain = 0.01:4, 
xmin = 0, 
xmax = 4.2, 
y=1cm,
x=1cm,
xtick={0,...,4},
ytick={-3,-2,-1,0,1,2,3,4},
enlarge y limits={rel=0.03}, 
enlarge x limits={rel=0.03}, 
ymin = -3.5, 
ymax = 4, 
after end axis/.code={\path (axis cs:0,0) node [anchor=north west,yshift=-0.075cm] {0} node [anchor=east,xshift=-0.075cm] {0};}]
\addplot[color=gray,line width = 1.0 pt, samples=400,domain=0:4] ({x},{2+(x-2)*0.5+2*0.86602540378*sqrt(x*(1-x/4))-x});
\addplot[color=red, line width = 1.0pt, samples=400,domain=0:4]({x},{2+(x-2)*0.5-2*0.86602540378*sqrt(x*(1-x/4))});
\addplot[color=blue, line width = 1.0pt, samples=400,domain=0:4]({x},{2+(x-2)*0.5-2*0.86602540378*sqrt(x*(1-x/4))-x});
\addplot[color=green, line width = 1.0pt, samples=400,domain=0:4]({x},{2+(x-2)*0.5+2*0.86602540378*sqrt(x*(1-x/4))});
\node[label={90:{$E_-$}},circle,fill,inner sep=1pt] at (axis cs:0.267949,0) {};
\node[label={90:{$\lambda_-$}},circle,fill,inner sep=1pt] at (axis cs:1,0) {};
\node[label={90:{$E_+$}},circle,fill,inner sep=1pt] at (axis cs:3.732051,0) {};
\node[label={90:{$\lambda_+$}},circle,fill,inner sep=1pt] at (axis cs:3,0) {};
\end{axis}
\end{tikzpicture}
\quad
\begin{tikzpicture}[domain=0.01:5]
\begin{axis}
[grid = major, 
axis x line = middle, 
axis y line = middle, 
xlabel={$x$}, 
xlabel style={at=(current axis.right of origin), anchor=west}, 
ylabel={$y$}, 
ylabel style={at=(current axis.above origin), anchor=south}, 
domain = 0.01:4, 
xmin = 0, 
xmax = 4.2, 
y=1cm,
x=1cm,
xtick={0,...,4},
ytick={-3,-2,-1,0,1,2,3,4},
enlarge y limits={rel=0.03}, 
enlarge x limits={rel=0.03}, 
ymin = -3.5, 
ymax = 4, 
after end axis/.code={\path (axis cs:0,0) node [anchor=north west,yshift=-0.075cm] {0} node [anchor=east,xshift=-0.075cm] {0};}]

\addplot[color=gray,line width = 1.0 pt, samples=400,domain=0:4] ({x},{2+(x-2)*(-0.5)+2*0.866*sqrt(x*(1-x/4))-x});
\addplot[color=red, line width = 1.0pt, samples=400,domain=0:4]({x},{2+(x-2)*(-0.5)-2*0.866*sqrt(x*(1-x/4))});
\addplot[color=blue, line width = 1.0pt, samples=400,domain=0:4]({x},{2+(x-2)*(-0.5)-2*0.866*sqrt(x*(1-x/4))-x});
\addplot[color=green, line width = 1.0pt, samples=400,domain=0:4]({x},{2+(x-2)*(-0.5)+2*0.866*sqrt(x*(1-x/4))});
\node[label={90:{$E_- = \lambda_+$}},circle,fill,inner sep=1pt] at (axis cs:1,0) {};
\node[label={90:{$E_+=\lambda_-$}},circle,fill,inner sep=1pt] at (axis cs:3,0) {};
\end{axis}
\end{tikzpicture}
\caption{\textcolor{red}{$g_{k;-}$}, \textcolor{green}{$g_{k;+}$}, \textcolor{blue}{$h_{k;-}$} and \textcolor{gray}{$h_{k;+}$} for $k=\pi /3$ (left) and $k=2\pi /3$ (right)}
\end{figure}
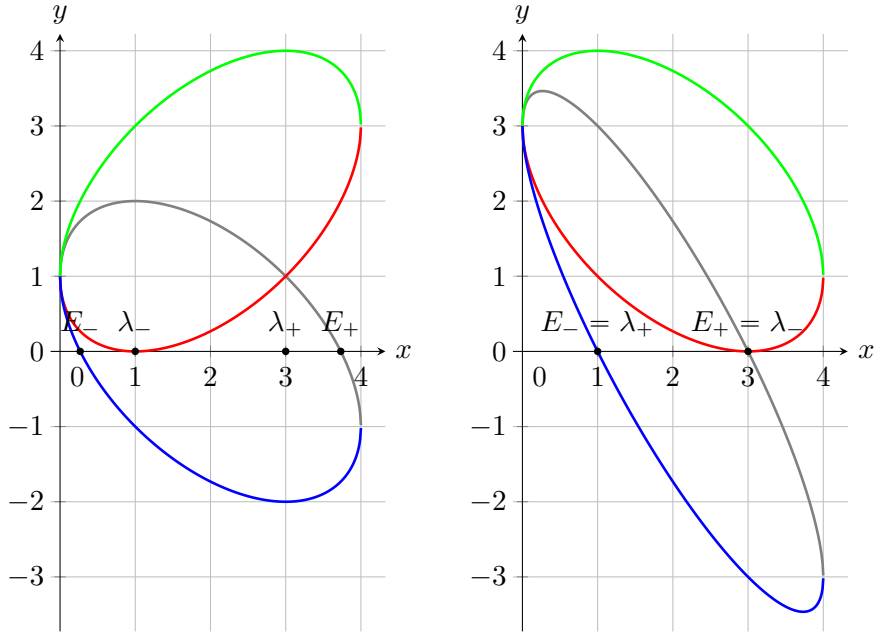



\noindent we may focus our analysis on $g_{k;-}$ and $h_{k;-}$. Define $\alpha(k) := (\cos(k)-1)(\sin(k))^{-1}$. 
The equation
\begin{equation*}
\label{funnytild}
h_{k;-}'(x) = (\cos(k)-1)-\sin(k)(-x+2)(x(4-x))^{-1/2} = 0
\end{equation*}
can be solved via the the quadratic formula and yields a single solution given by 
\begin{equation*}
\label{sunnymaybe}
\begin{cases}
2 + 2\sqrt{1-(1+\alpha^2(k))^{-1}} & \text{if} \ k \in (0,\pi) \\
2 - 2\sqrt{1-(1+\alpha^2(k))^{-1}} & \text{if} \ k \in (\pi,2\pi).
\end{cases}
\end{equation*}
Consequently $h_{k;-}$ has exactly one local extremum. Combining this with the fact that $h_{k;-}$ is continuous, $h_{k;-}(0) = 2-2\cos(k) > 0$ and $h_{k;-}(4)=-2+2\cos(k)<0$, we conclude that $E_-(k)$ is the only root of $h_{k;-}$. By \eqref{symmetry} we immediately get that $E_{+}(k)$ is the only root of $h_{k;+}$. We move on to the analysis of $g_{k;-}$. The equation 
\begin{equation*}
g_{k;-}'(x) = \cos(k)-\sin(k)(-x+2)(x(4-x))^{-1/2} = 0 \end{equation*} 
has a single solution given by
\begin{equation*}
\lambda_{-}(k) := \begin{cases}
2 - 2\sqrt{1-(1+\beta^2(k))^{-1}} = 2 - 2|\cos(k)| & \text{if} \ k \in (0,\pi/2] \cup (\pi,3\pi/2] \\
2 + 2\sqrt{1-(1+\beta^2(k))^{-1}} = 2 + 2|\cos(k)| & \text{if} \ k \in [\pi/2,\pi) \cup [3\pi/2,2\pi).
\end{cases}
\end{equation*} 
Here $\beta(k) := \cot(k)$. We conclude that $g_{k;-}$ has exactly one local extremum. We note that $g_{k;-}(\lambda_-(k)) = 0$ when $k \in (0,\pi)$ and $g_{k;-}(\lambda_-(k)) = 4$ when $k \in (\pi,2\pi)$. Finally, we have 
\begin{equation*}
h_{k;-}''(x) = g_{k;-}''(x) = 4\sin(k)(x(4-x))^{-3/2}.
\end{equation*} 
We summarize the relevant details in Tables \ref{tab:g1}, \ref{tab:g2}, \ref{tab:g3} and \ref{tab:g4}. We are ready to complete the proof.
\begin{table}[htbp]
  \centering
    \extrarowsep=_3pt^3pt
     \begin{tabu}to\linewidth{|[1.0pt gray]c|c|c|c|c|[1.0pt gray]}
    \tabucline[1.0pt gray]-
    $k$   & $\lambda_-(k)$   & $E_-(k)$ & & $g_{k;-}''(x) = h_{k;-}''(x)$  \\
    \tabucline[1.0pt gray]-
    $\in (0, \pi/2]$   & $=2-2\cos(k)$  & $= 2-2\cos(k/2)$ & $E_-(k) < \lambda_-(k)$ & $>0 \ \forall x \in [0,4]$  \\\hline
    $\in (\pi/2,\pi)$   & $= 2-2\cos(k)$ & = $2-2\cos(k/2)$ & $E_-(k) < \lambda_-(k)$ & $>0 \ \forall x \in [0,4]$ \\\hline
    $\in (\pi, 3\pi/2)$  & $= 2+2\cos(k)$  & = $2-2\cos(k/2)$ & $\lambda_-(k) < E_-(k)$ & $<0 \ \forall x \in [0,4]$ \\\hline
    $\in [3\pi/2, 2\pi)$  & $= 2+2\cos(k)$  & = $2-2\cos(k/2)$ & $\lambda_-(k) < E_-(k)$ & $<0 \  \forall x \in [0,4]$ \\
    \tabucline[1.0pt gray]-
    \end{tabu}%
 \caption{Analysis of $g_{k;-}$ and $h_{k;-}$ for different values of $k$}
\label{tab:g1}%
\end{table}%
\begin{table}[htbp]
  \centering
  \begin{tabu}to\linewidth{|[1.0pt gray]c| c c c c c c c|[1.0pt gray]}
    \tabucline[1.0pt gray]-
    $x$  & 0 & & $E_-(k)$ & & $\lambda_-(k)$ & & 4 \\
    \tabucline[1.0pt gray]-
    $g_{k;-}(x)$  & & $\searrow$ & & $\searrow$ & 0 & $\nearrow$ &   \\\hline
    $h_{k;-}(x)$  & & $+$ & 0 & $-$ & & $-$ &  \\
    \tabucline[1.0pt gray]-
    \end{tabu}%
     \begin{tabu}to\linewidth{|[1.0pt gray]c| c c c c c c c|[1.0pt gray]}
    \tabucline[1.0pt gray]-
    $x$  & 0 & & $\lambda_-(k)$ & & $E_-(k)$ & & 4 \\
    \tabucline[1.0pt gray]-
    $g_{k;-}(x)$  & & $\nearrow$ & 4 & $\searrow$ &  & $\searrow$ &   \\\hline
    $h_{k;-}(x)$  & & $+$ & & $+$ & 0 & $-$ &  \\
    \tabucline[1.0pt gray]-
    \end{tabu}%
 \caption{Variations of $g_{k;-}$ and sign of $h_{k;-}$ for $k \in (0, \pi)$ (left) and $k \in (\pi, 2\pi)$ (right)}
\label{tab:g2}%
\end{table}%

\noindent \underline{Case $k \in (0,\pi)$, $y=0$:} Depending on $E \in [0,4] \setminus \{E_-(k)\}$, we show that there exists an interval $\mathcal{I} \ni E$ such that one of the two following hold: 
\begin{align}
\label{a11}
\mathcal{I} & < g_{k;-}(\mathcal{I}), \\
\label{a12}  
\mathcal{I} & > g_{k;-}(\mathcal{I}).
\end{align}

\noindent \textbf{(A)} For $E \in [0,E_-(k))$, there is $\epsilon > 0$ such that $E+ \epsilon < g_{k;-}(E+\epsilon)$. Thus \eqref{a11} holds for $\mathcal{I} = (E-\epsilon,E+\epsilon)$. \textbf{(B)} For $E \in (E_-(k),\lambda_-(k))$, there is $\epsilon > 0$ such that $g_{k;-}(E-\epsilon) < E-\epsilon$. Thus \eqref{a12} holds for $\mathcal{I}=(E-\epsilon,E+\epsilon)$. \textbf{(C)} For $E = \lambda_-(k)$, there is $\epsilon_1 >0$ such that $g_{k;-}(\lambda_-(k)-\epsilon_1) < \lambda_-(k)-\epsilon_1$. Thus $[0,g_{k;-}(\lambda_-(k)-\epsilon_1)] = g_{k;-}([\lambda_-(k)-\epsilon_1,\lambda_-(k)]) < [\lambda_-(k)-\epsilon_1,\lambda_-(k)]$. By continuity of $g_{k;-}$ there is $\epsilon_2 > 0$ such that $g_{k;-}\left([\lambda_-(k), \lambda_-(k)+\epsilon_2]\right) = [0, g_{k;-}(\lambda_-(k)+\epsilon_2)] \subset [0, g_{k;-}(\lambda_-(k)-\epsilon_1)]$. Thus \eqref{a12} holds for $\mathcal{I} = (\lambda_-(k)-\epsilon_1, \lambda_-(k)+\epsilon_2)$.
\textbf{(D)} Finally for $E \in (\lambda_-(k),4]$, there is $\epsilon >0$ such that $h_{k;-}(t) < -2\epsilon$ for all $t \in [E-\epsilon,E+\epsilon]$, and so $g_{k;-}(E+\epsilon) < E-\epsilon$. Thus \eqref{a12} holds for $\mathcal{I} = (E-\epsilon,E+\epsilon)$.

\noindent \underline{Case $k \in (0,\pi)$, $y=1$:} We denote $\lambda_+(k) = 4 - \lambda_-(k)$ the location of the extremum of $g_{k;+}$. Depending on $E \in [0,4] \setminus \{E_+(k)\}$, one procedes in the same fashion as before to show that there exists an interval $\mathcal{I} \ni E$ such that one of the two following hold: 
\begin{align}
\label{a13}
\mathcal{I} & < g_{k;+}(\mathcal{I}), \\
\label{a14}  
\mathcal{I} & > g_{k;+}(\mathcal{I}).
\end{align}
\begin{table}[htbp]
  \centering
    \extrarowsep=_3pt^3pt
     \begin{tabu}to\linewidth{|[1.0pt gray]c|c|c|c|c|[1.0pt gray]}
    \tabucline[1.0pt gray]-
    $k$  & $\lambda_+(k)$   & $E_+(k)$ & & $g_{k;+}''(x) = h_{k;+}''(x)$  \\
    \tabucline[1.5pt gray]-
    $\in (0, \pi/2]$    & $= 2+2\cos(k)$  & $= 2+2\cos(k/2)$ & $\lambda_+(k) < E_+(k)$ & $<0 \ \forall x \in [0,4]$  \\\hline
    $\in (\pi/2,\pi)$  & $= 2+2\cos(k)$ & = $2+2\cos(k/2)$ & $\lambda_+(k) < E_+(k)$ & $<0 \ \forall x \in [0,4]$ \\\hline
    $\in (\pi, 3\pi/2)$  & $= 2-2\cos(k)$  & = $2+2\cos(k/2)$ & $E_+(k) < \lambda_+(k)$ & $>0 \ \forall x \in [0,4]$ \\\hline
    $\in [3\pi/2, 2\pi)$  & $= 2-2\cos(k)$  & = $2+2\cos(k/2)$ & $E_+(k) < \lambda_+(k)$ & $>0 \  \forall x \in [0,4]$ \\
    \tabucline[1.0pt gray]-
    \end{tabu}%
 \caption{Analysis of $g_{k;+}$ and $h_{k;+}$ for different values of $k$}
\label{tab:g3}%
\end{table}%
\begin{table}[htbp]
  \centering
  \begin{tabu}to\linewidth{|[1.0pt gray]c| c c c c c c c|[1.5pt gray]}
    \tabucline[1.0pt gray]-
    $x$  & 0 & & $\lambda_+(k)$ & & $E_+(k)$ & & 4 \\
    \tabucline[1.0pt gray]-
    $g_{k;+}(x)$  & & $\nearrow$ & 4 & $\searrow$ & & $\searrow$ &   \\\hline
    $h_{k;+}(x)$  & & $+$ & & $+$ & 0 & $-$ &  \\
    \tabucline[1.0pt gray]-
    \end{tabu}%
     \begin{tabu}to\linewidth{|[1.0pt gray]c| c c c c c c c|[1.0pt gray]}
    \tabucline[1.0pt gray]-
    $x$  & 0 & & $E_+(k)$ & & $\lambda_+(k)$ & & 4 \\
    \tabucline[1.0pt gray]-
    $g_{k;+}(x)$  & & $\searrow$ & & $\searrow$ & 0 & $\nearrow$ &   \\\hline
    $h_{k;+}(x)$  & & $+$ & 0 & $-$ & & $-$ &  \\
    \tabucline[1.0pt gray]-
    \end{tabu}%
 \caption{Variations of $g_{k;+}$ and sign of $h_{k;+}$ for $k \in (0, \pi)$ (left) and $k \in (\pi, 2\pi)$ (right)}
\label{tab:g4}%
\end{table}%
The case of $k \in (\pi,2\pi)$ is also covered because $g_{(2\pi-k);+}(x) = g_{k;-}(x)$ for all $k \in \mathcal{T}$.
\qed

\section{The Multi-dimensional Case}
\label{ddim}

We introduce the tensor product notation. The position space is the Hilbert space $\mathcal{H} = \ell^2(\Z^d) \approx \otimes _{i=1}^d \ell^2(\Z)$. The $d$-dimensional Laplacian is equivalent to
\begin{equation*}
\Delta  \approx \Delta_1 \otimes \1 \otimes ... \otimes \1 \ + \ \1 \otimes \Delta_2 \otimes ... \otimes \1 \ + \ ... \ + \ \1 \otimes  ... \otimes \1 \otimes \Delta_d
\end{equation*} 
where the $\Delta_i$ are copies of the one-dimensional Laplacian. 
The potentials $W$ and $V$ cannot be written explicitely in tensor product notation, whereas $W'$ can. The generator of dilations is
\begin{equation*}
A \approx A_1 \otimes \1 \otimes ... \otimes \1 \ + \ \1 \otimes A_2 \otimes ... \otimes \1 \ + \ ... \ + \ \1 \otimes  ... \otimes \1 \otimes A_d, \quad \mathcal{D}(A) := \otimes_{i=1}^d \mathcal{D}(A_i)
\end{equation*}
where the $A_i$ are copies of the $1d$ generator of dilations defined as the closure of \eqref{do34}. Since the copies $A_i$ are all self-adjoint, $A$ is self-adjoint.


\subsection{\texorpdfstring{$\mathcal{C}^1(A)$}{CAA} Regularity}
It is immediate that $[\Delta, \i A]$ extends to a bounded form and 
\begin{equation*}
\label{oieoie22}
[\Delta, \i A] _{\circ} \approx \Delta_1(4-\Delta_1) \otimes \1 ... \otimes \1 \ + \ \1 \otimes \Delta_2(4-\Delta_2) \otimes ... \otimes \1 \ + \ ... \ +\ \1 \otimes \ ... \ \otimes \1 \otimes \Delta_d(4-\Delta_d).
\end{equation*}
By induction, we have that $\Delta \in \mathcal{C}^{\infty}(A)$. We turn to the potential $W$.
Define 
\begin{align}
K_W &:= \label{Kdun} 2^{-1} W \sum _{i=1} ^d (S_i^* +S_i) + 2^{-1} \sum _{i=1} ^d (S_i^* +S_i) W, \\
\label{bw6}
B_W &:= \sum _{i=1}^d U_i \tilde{W} (S_i^*-S_i) - \sum _{i=1}^d (S_i^*-S_i) \tilde{W} U_i,
\end{align}
where the $U_i$ are the operators $(U_i u)(n) := n_i |n|^{-1} u(n)$ and $\tilde{W}$ is the operator $(\tilde{W}u)(n) := q \sin(k(n_1+...+n_d))u(n)$. We have that for all $u,v \in \ell_0(\Z^d)$,
\begin{equation*}
    \langle u, [W,\i A] v \rangle = \langle u, K_W v \rangle + \langle u, B_W v \rangle. 
\end{equation*}
Thus $[W, \i A]$ extends to a bounded operator and $[W, \i A]_{\circ} = K_W + B_W$. For the potential $W'$,
\begin{align*}
[W', \i A]_{\circ} & \approx [W_1', \i A_1]_{\circ} \otimes W_2' \otimes ...\otimes W_d' \ + \ W_1' \otimes [W_2', \i A_2]_{\circ} \otimes ...\otimes W_d' \\
& \quad \quad + \ ... \ + \  W_1' \otimes ...\otimes W_{d-1}' \otimes [W_d', \i A_d]_{\circ} \\
&:= K_{W'} + B_{W'}
\end{align*}
where 
\begin{equation}
K_{W'} := K_{W_1'} \otimes W_2' \otimes ...\otimes W_d' \ + \ W_1' \otimes K_{W_2'} \otimes ...\otimes W_d' \ + \ ... \ + \  W_1' \otimes ...\otimes W_{d-1}' \otimes K_{W_d'}, 
\end{equation}
\begin{equation}
\label{BW2}
B_{W'} := B_{W_1'} \otimes W_2' \otimes ...\otimes W_d' \ + \ W_1' \otimes B_{W_2'} \otimes ...\otimes W_d' \ + \ ... \ + \  W_1' \otimes ...\otimes W_{d-1}' \otimes B_{W_d'},  
\end{equation}
\begin{equation}
\label{BWd}
K_{W_i'} = 2^{-1}W_i' (S_i^*+S_i) +2^{-1}(S_i^*+S_i)W_i', \quad \text{and} \quad
B_{W_i'} = \tilde{W_i'} (S_i^*-S_i) - (S_i^*-S_i) \tilde{W_i'}.
\end{equation} 
Here $W_i'$ and $\tilde{W_i'}$ are one-dimensional operators defined by $(W_i'u)(n) = q_i \sin(k_i n) n^{-1} u(n)$, and $(\tilde{W_i'}u)(n) := q_i \sin(k_i n) u(n)$. Note that $K_W$ and $K_{W'}$ are compact, while $B_W$ and $B_{W'}$ are bounded but not compact by Proposition \ref{notC1u2}. As for the form $[V, \i A]$, we have as in \eqref{Goddy} that for all $u,v \in \ell_0(\Z^d)$,
\begin{equation}
\label{Goddy2}
\langle u,[V, \i A] v \rangle = - \sum _{i=1}^d \langle u, \big[(N_i -2^{-1})(V-\tau_iV)S_i + (N_i -2^{-1})(V-\tau_i^*V)S_i^*\big] v \rangle.
\end{equation}
Hypothesis \eqref{kdkd2} allows us to extend $[V, \i A]$ into a compact operator. This leads to the following

\begin{proposition}
\label{oieoie2}
$H$ and $H'$ are of class $\mathcal{C}^1(A)$.
\end{proposition}
As in the one-dimensional case, we have
\begin{proposition}
\label{notC1u2}
$H$ and $H'$ are not of class $\mathcal{C}^{1,\text{u}}(A)$.
\end{proposition}
\proof
As in the proof of Proposition \ref{NotC1u}, one shows that $B_W$ and $B_{W'}$ are not compact. This can be done by considering the sequence $(\delta_j)_{j \geq 2}$ of unit vectors in $\ell^2(\Z^d)$ satisfying $(\delta_j)(n)= \delta_{j;n_1} \delta_{0;n_2} \cdot \cdot \cdot \delta_{0;n_d}$. This sequence is converging weakly to zero. If $B_{W'}$ was compact, we would require $B_{W'} \delta_j$ to converge strongly to zero, but this would lead to the same contradiction as in Proposition \ref{NotC1u}. As for $B_W$, we commute $U_i$ with $(S_i^*-S_i)$ to produce a compact and get 
\begin{equation*}
    B_W = \sum _{i=1}^d U_i [ \tilde{W}(S_i^*-S_i) - (S_i^*-S_i)\tilde{W}] + \ \text{compact}.
\end{equation*}
Again, applying this operator to $\delta_j$ and requiring the limit to converge strongly to zero would generate the same contradiction.
\qed


\subsection{Classical Mourre Theory}
Recall that $\sigma (\Delta) = \overline{\sigma(\Delta_1)+...+\sigma(\Delta_d)} = [0,4d]$.
We would like to identify the sub-intervals of $\sigma(\Delta)$ for which a strict Mourre estimate for $\Delta$ holds. Recall the function $\rho_{T}^A(E)$ introduced in \eqref{VARRHO}. In the setting of the tensor product of two operators we have the standard result \cite[Theorem 8.3.6]{ABG} : 
\begin{equation}
\varrho_{T}^{A}(E) = \inf \limits_{E = x_1+x_2} [ \varrho_{T_1}^{A_1}(x_1)+\varrho_{T_2}^{A_2}(x_2)],
\label{apop}
\end{equation}
where $T := T_1 \otimes \1 + \1 \otimes T_2$ and $A := A_1 \otimes \1 + \1 \otimes A_2$ are an arbitrary pair of conjugate self-adjoint operators. As a consequence of \eqref{varrhoDelta}, we infer that in the case of $d=2$, $0< \varrho_{\Delta}^{A}(E) < \infty$ if and only if $E \in (0,8)\setminus \{4\}$, so that the strict Mourre estimate for $\Delta$ holds at every point of the spectrum of $\Delta$, except at the critical points $\{0,4,8\}$. If $d>2$, then a similar formula to \eqref{apop} holds with nested terms. One easily sees that $0 < \varrho_{\Delta}^{A}(E) < \infty$ if and only if $E \in (0,4d)\setminus \{4j\}_{j=1}^{d-1}$, so that the strict Mourre estimate holds at every point of the spectrum of $\Delta$, except at the critical points $\{4j\}_{j=0}^{d}$. For the special case of the discrete Laplacian, the classic strict Mourre estimate can be derived without resorting to \eqref{apop} whose proof is somewhat elaborate. We show how this can be done.

We work in two dimensions, but remark that the same setup can be generalized for $d>2$. Let $\epsilon \in (0,2)$ be given and let $\mathcal{I} = (\epsilon, 4-\epsilon)$. By \eqref{oieoie}, we have 
\begin{equation}
\label{idfdf}
E_{\mathcal{I}}(\Delta_i)[\Delta_i, \i A_i]_{\circ}E_{\mathcal{I}_i}(\Delta_i) \geqslant \epsilon(4-\epsilon) E_{\mathcal{I}}(\Delta_i).
\end{equation}
The following Proposition converts the one dimensional (optimal) strict Mourre estimate for $\Delta$ into a two dimensional strict Mourre estimate.
\begin{proposition}
\label{sunnybunny}
For every $\epsilon \in (0,2)$, let $\mathcal{I}:= (\epsilon, 4-\epsilon)$, or $\mathcal{I} := (4+\epsilon, 8-\epsilon)$. Then the strict Mourre estimate holds for the two-dimensional Laplacian $\Delta$ on $\mathcal{I}$, namely:
\begin{equation}
E_{\mathcal{I}}(\Delta)[\Delta, \i A]_{\circ}E_{\mathcal{I}}(\Delta) \geqslant \epsilon(4-\epsilon) E_{\mathcal{I}}(\Delta).
\label{rockstar}
\end{equation}
\end{proposition}

\proof
We consider the case $\mathcal{I} = (\epsilon,4-\epsilon)$, as the other case is similar. Note that $\chi_{\mathcal{I}}(x_1+x_2)$ is supported on the open set $U := \{(x_1,x_2) \in [0,4] \times [0,4] : x_1+x_2 \in \mathcal{I} \}$ which has the form of a trapezoid. We decompose $U$ in four regions, namely $U_{(1,1)} : = U \cap [0, \epsilon)\times [0, \epsilon)$, $U_{(1,2)} := U \cap [0, \epsilon) \times [\epsilon, 4-\epsilon)$, $U_{(2,1)} := U \cap [\epsilon, 4-\epsilon) \times [0, \epsilon)$, and $U_{(2,2)} := U \cap [\epsilon, 4-\epsilon) \times [\epsilon, 4-\epsilon)$.

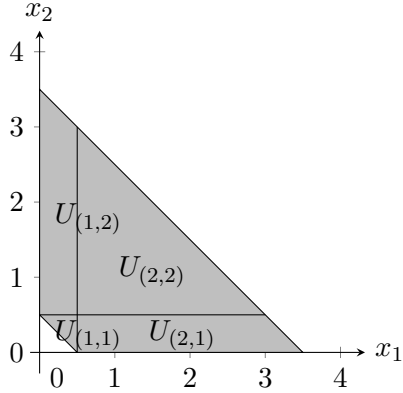
\begin{figure}
\begin{tikzpicture}[domain=0.01:4]
\begin{axis}
[grid = none, 
axis x line = middle, 
axis y line = middle, 
xlabel={$x_1$}, 
xlabel style={at=(current axis.right of origin), anchor=west}, 
ylabel={$x_2$}, 
ylabel style={at=(current axis.above origin), anchor=south}, 
domain = 0.01:4, 
xmin = 0, 
xmax = 4.2, 
y=1cm,
x=1cm,
xtick={0,...,4},
ytick={0,1,2,3,4},
enlarge y limits={rel=0.07}, 
enlarge x limits={rel=0.03}, 
ymin = 0, 
ymax = 4, 
after end axis/.code={\path (axis cs:0,0) node [anchor=north west,yshift=-0.075cm] {0} node [anchor=east,xshift=-0.075cm] {0};}]
\addplot[domain=0:3.5,fill=gray!50] {3.5-x}\closedcycle;
\addplot[domain=0:0.5,fill=white!50] {0.5-x}\closedcycle;
\addplot +[mark=none,color=black] coordinates {(0.5, 0) (0.5, 3)};
\addplot +[mark=none,color=black] coordinates {(0, 0.5) (3, 0.5)};
\node[label={90:{$U_{(1,1)}$}}] at (axis cs:0.65,-0.25) {};
\node[label={90:{$U_{(1,2)}$}}] at (axis cs:0.65,1.3) {};
\node[label={90:{$U_{(2,2)}$}}] at (axis cs:1.5,0.6) {};
\node[label={90:{$U_{(2,1)}$}}] at (axis cs:1.9,-0.25) {};
\end{axis}
\end{tikzpicture}
\caption{Support of $\chi_{\mathcal{I}}(x_1+x_2)$ for $\mathcal{I} = (0.5,3.5)$.}
\end{figure}
For $n\in \N$ and $(i,j) \in \{1,...,2^n\}\times\{1,...,2^n\}$, consider the disjoint intervals of the form 
\begin{equation*}
\mathcal{I}_{1;i;n} := \left[(i-1)2^{-n}\epsilon,i2^{-n}\epsilon \right) \ \ \ \text{and} \ \ \ \mathcal{I}_{2;j;n} := \left[\epsilon+(j-1)2^{-n}(4-2\epsilon),\epsilon+j2^{-n}(4-2\epsilon)\right) 
\end{equation*}
which satisfy $\cup_{i=1}^{2^n} \mathcal{I}_{1;i;n}  = [0,\epsilon)$ and $\cup_{j=1}^{2^n} \mathcal{I}_{2;j;n}  = [\epsilon,4-\epsilon)$.
For $\alpha,\beta \in \{1,2\}$, let 
\begin{equation*}
F_{\alpha,\beta,n} := \{(i,j) \in \{1,...,2^n\}\times \{1,...,2^n\} : \mathcal{I}_{\alpha;i;n} \times \mathcal{I}_{\beta;j;n} \subset U_{(\alpha,\beta)} \}.
\end{equation*}
Then 
\begin{equation*}
\lim \limits_{n\to \infty} \bigcup \limits_{(i,j)\in F_{\alpha,\beta,n}} \mathcal{I}_{\alpha;i;n} \times \mathcal{I}_{\beta;j;n} = U_{(\alpha,\beta)}.
\end{equation*}
In terms of operators, we have 
\begin{equation*}
\slim \limits_{n\to \infty} \sum \limits_{\alpha,\beta=1,2} \sum \limits_{(i,j)\in F_{\alpha,\beta,n}} E_{\mathcal{I}_{\alpha;i;n}}(\Delta_1) \otimes E_{\mathcal{I}_{\beta;j;n}}(\Delta_2) = E_{\mathcal{I}}(\Delta).
\end{equation*}
Now, $[\Delta, \i A]_{\circ} = [\Delta_1, \i A_1]_{\circ} \otimes \1 + \1 \otimes [\Delta_2, \i A_2]_{\circ}$, so for fixed $n$ we calculate:
\begin{align*}
& \left( \sum \limits_{\alpha,\beta} \sum \limits_{(i,j)} E_{\mathcal{I}_{\alpha;i;n}}(\Delta_1) \otimes E_{\mathcal{I}_{\beta;j;n}}(\Delta_2) \right) [\Delta, \i A ]_{\circ} \left( \sum \limits_{\alpha',\beta'} \sum \limits_{(i',j')} E_{\mathcal{I}_{\alpha';i';n}}(\Delta_1) \otimes E_{\mathcal{I}_{\beta';j';n}}(\Delta_2)  \right) \\
& = \sum \limits_{\alpha,\beta} \sum \limits_{(i,j)} \sum \limits_{\alpha'} \sum \limits_{(i',j)} E_{\mathcal{I}_{\alpha;i;n}}(\Delta_1) [\Delta_1, \i A_1]_{\circ} E_{\mathcal{I}_{\alpha';i';n}}(\Delta_1) \otimes E_{\mathcal{I}_{\beta;j;n}}(\Delta_2) \\
& \ \ \ \ \ \ \ \ \ \ \ \ \ + \sum \limits_{\alpha,\beta} \sum \limits_{(i,j)} \sum \limits_{\beta'} \sum \limits_{(i,j')}  E_{\mathcal{I}_{\alpha;i;n}}(\Delta_1) \otimes E_{\mathcal{I}_{\beta;j;n}}(\Delta_2) [\Delta_2, \i A_2]_{\circ} E_{\mathcal{I}_{\beta';j';n}}(\Delta_2) \end{align*}
\begin{align*}
& \geqslant  \sum \limits_{\alpha,\beta} \sum \limits_{(i,j)} E_{\mathcal{I}_{\alpha;i;n}}(\Delta_1) [\Delta_1, \i A_1]_{\circ} E_{\mathcal{I}_{\alpha;i;n}}(\Delta_1) \otimes E_{\mathcal{I}_{\beta;j;n}}(\Delta_2) \\
& \ \ \ \ \ \ \ \ \ \ \ \ \ + \sum \limits_{\alpha,\beta} \sum \limits_{(i,j)} E_{\mathcal{I}_{\alpha;i;n}}(\Delta_1) \otimes E_{\mathcal{I}_{\beta;j;n}}(\Delta_2) [\Delta_2, \i A_2]_{\circ} E_{\mathcal{I}_{\beta;j;n}}(\Delta_2) \\
& \geqslant  \sum \limits_{\alpha,\beta} \sum \limits_{(i,j)} (c_{\alpha;i;n}+c_{\beta;j;n})E_{\mathcal{I}_{\alpha;i;n}}(\Delta_1) \otimes E_{\mathcal{I}_{\beta;j;n}}(\Delta_2) 
\end{align*}
for some positive constants $c_{\alpha;i;n}$ and $c_{\beta;j;n}$ which can possibly be $0$ if $\alpha=1$ and $i=1$ or if $\beta=1$ and $j=1$. However $c_{\alpha;i;n}$ and $c_{\beta;j;n}$ are not independent since $(i,j) \in F_{\alpha,\beta,n}$; in fact $c_{\alpha;i;n}+c_{\beta;j;n} \geqslant \epsilon (4-\epsilon) > 0$ for all $\alpha,\beta \in \{1,2\}$ and $(i,j) \in F_{\alpha,\beta,n}$. The case $\alpha=\beta=1$ is the least obvious. Consider $C(x_1,x_2) = x_1(4-x_1) + x_2(4-x_2)$ defined for $(x_1,x_2) \in [0,\epsilon) \times [0,\epsilon)$ which represents how $c_{\alpha;i;n}+c_{\beta;j;n}$ varies. Then $c_{1;i;n}+c_{1;j;n} \geqslant C(x_1,\epsilon-x_1) = -2x_1^2 + 2x_1\epsilon-\epsilon^2+4\epsilon \geqslant -\epsilon^2+4\epsilon$. The proof is now complete by taking the limit $n\to \infty$. 
\qed






We are now working our way towards a classic Mourre estimate \eqref{Mourre K} for the full Schr{\"o}dinger operator $H$. As in the one-dimensional case, $[V,\i A]_{\circ}$ is compact, and $[W,\i A]_{\circ}$ is the sum of a compact operator $K_W$ and a bounded operator $B_{W}$ defined by \eqref{bw6}, so we really only have to show that $E_{\mathcal{I}}(H)B_{W}E_{\mathcal{I}}(H)$ is compact. Let $k \in \mathcal{T} := \R \setminus \pi \Z \ (\text{mod} \ 2\pi)$. 
\begin{proposition}
\label{sj9}
Let 
\begin{equation}
\label{Lambdas}
    E(k) := \begin{cases}
    4-4\cos(k/2) & \ \text{for} \ k \in (0,\pi) \\
    4+4\cos(k/2) & \ \text{for} \ k \in (\pi,2\pi)
    \end{cases} \quad \text{and} \quad \mu(H) := [0,E(k)) \cup (4d-E(k),4d]. 
\end{equation}
For each $E \in \mu(H)$ there exists $\epsilon(E) > 0$ such that for all $\theta \in C_c^{\infty}(\R)$ supported on $\mathcal{I} := (E-\epsilon, E+\epsilon)$, $\theta(\Delta) \tilde{W} \theta(\Delta) = 0$. In particular, $\theta(\Delta) B_W \theta(\Delta)$ is compact. Consequently, for every $E \in \mu(H)$, the classical Mourre estimate \eqref{Mourre K} holds for $H$ on $\mathcal{I}'$, where $\overline{\mathcal{I}'} \subset \mathcal{I}$.
\end{proposition}
\begin{remark}
\label{kol}
The unitary transformation $u(n) \mapsto (-1)^{n_1+...+n_d}u(n)$ for all $u \in \mathcal{H} := \ell^2(\Z^d)$ shows that $\Delta$ and $4d-\Delta$ are unitarily equivalent, (and likewise for $H:= \Delta+W+V$ and $4d-\Delta + W + V$). Because of this symmetry, showing that $\theta(\Delta)B_W\theta(\Delta)$ is compact for $\theta$ supported on $\mathcal{I}=(E-\epsilon,E+\epsilon)$ and $E \in [0,E(k))$ implies it for $E \in (4d-E(k),4d]$ (and vice versa). This symmetry is due to the bipartite structure of $\Z^d$. 
\end{remark} 
\begin{remark}
That $\theta(\Delta) B_W \theta(\Delta)$ is compact and not zero is because commuting $U_i$ with $\Delta$ produces a compact operator. Then using the strict Mourre estimate for $\Delta$ from Proposition \ref{sunnybunny}, one derives the Mourre estimate for $H$ in the same way as in Proposition \ref{goosh}. 
\end{remark}
\begin{proof}
The strategy is the same as in 1d (cf.\ Lemma \ref{ImportantLemma} for the notation). Thanks to \eqref{vd=bv54}, 
\begin{equation}
\theta(\Delta)\tilde{W}\theta(\Delta) = \theta(\Delta)\theta \left( \sum_{i=1}^d g_k(\Delta_i, \widecheck{\1}_{[0,\pi],i})\right)qT_k/(2\i) - \theta(\Delta)\theta \left( \sum_{i=1}^d g_{2\pi-k}(\Delta_i, \widecheck{\1}_{[0,\pi],i})\right)qT_{2\pi-k}/(2\i),
\end{equation} 
and so it is enough to show that $\theta(\Delta)\theta \left(\sum_i g_k(\Delta_i, \widecheck{\1}_{[0,\pi],i})\right) = 0$ for $k \in \mathcal{T}$ and $\theta$ appropriately chosen. We appeal to the functional calculus for self-adjoint commuting operators. Consider the function $g_k(x,y)$ of \eqref{okid} defined for $(x,y) \in \sigma(\Delta_i)\times\sigma(\widecheck{\1}_{[0,\pi],i}) = [0,4]\times\{0,1\}$. We want to find $\epsilon(E) > 0$ such that for the interval $\mathcal{I} := (E-\epsilon,E+\epsilon)$ we have
\begin{equation}
\mathcal{I} \cap \big \{ \sum_{1 \leqslant i \leqslant d} g_k(x_i,y_i) : (x_1,...,x_d) \in R \ \ \text{and} \ \  (y_1,...,y_d) \in \{0,1\}^d \big \} = \emptyset
\label{oksy}
\end{equation} 
where $R$ is the region defined by $R := \{(x_1,...,x_d) \in [0,4]^d : x_1+...+x_d \in \mathcal{I} \}$. In this way if supp$(\theta)=\mathcal{I}$, then we will have $\theta(x_1+...+x_d)\theta(\sum_i g_k(x_i,y_i)) = 0$ as required. Set 
\begin{multline}
\label{verify}
\mathcal{E}_d (k) := \{ E \in [0,4d] : \ \text{there exist} \ (x_1,...,x_d) \in [0,4]^d \ \text{and} \ (y_1,...,y_d) \in \{0,1\}^d \\
\text{such that} \ E = x_1+...+x_d = g_k(x_1,y_1) +...+ g_k(x_d,y_d) \}.   
\end{multline}
If $E \in \mathcal{E}_d (k)$, then \eqref{oksy} does not hold at $E$. Note also that $\mathcal{E}_d (k) = \mathcal{E}_d (2\pi-k)$. First we work in $d=2$, and extend the result for $d \geqslant 3$ at the very end. To identify the set $\mathcal{E}_2(k)$, we solve 
\begin{equation}
\label{espin1}
\mathscr{E}_{k;\ast;\diamond} : \quad  h_{k;\ast}(x_1)+h_{k;\diamond}(x_2) = 0, \quad \text{for} \ \ast,\diamond \in \{-,+\}.
\end{equation} 
We denote by $S_{k;\ast;\diamond}$ the solutions to $\mathscr{E}_{k;\ast;\diamond}$ and let $E_{k;\ast;\diamond} := \{x_1+x_2: (x_1,x_2) \in S_{k;\ast;\diamond}\}$. By \eqref{symmetry}, $(x_1,x_2) \in S_{k;-;-}$ iff $(4-x_1,4-x_2) \in S_{k;+;+}$. By symmetry, $(x_1,x_2) \in S_{k;-;+}$ iff $(x_2,x_1) \in S_{k;+;-}$. We focus first on $\mathscr{E}_{k;-;-}$. In this case, note that $(x_1,x_2)$ is a solution iff $(x_2,x_1)$ is a solution. With the change of variables $(x_1,x_2) = (2-2\cos(\phi),2-2\cos(\varphi))$, $(\phi,\varphi) \in [0,\pi]^2$, $\mathscr{E}_{k;-;-}$ becomes
\begin{equation*}
-2\cos(\phi)(\cos(k)-1)-2\sin(k)\sin(\phi) - 2\cos(\varphi)(\cos(k)-1)-2\sin(k)\sin(\varphi) = 0
\end{equation*}
which reduces to
\begin{equation}
\label{SLD1}
-8\sin(k/2)\sin((\phi +\varphi -k)/2)\cos((\phi-\varphi)/2) = 0.
\end{equation}
Thus $(\phi+\varphi-k)/2 =0 \ [\text{mod} \ \pi]$ or $(\phi-\varphi)/2 = \pi/2 \ [\text{mod} \ \pi]$. Considering $(\phi,\varphi) \in [0,\pi]^2$ and the cases $k \in (0,\pi)$ and $k \in (\pi,2\pi)$ separately, one can rule out several possibilities. Let $J_k := [0,k]$ if $k \in (0,\pi)$, and $J_k := [k-\pi,\pi]$ if $k \in (\pi,2\pi)$. The valid solutions of the previous equation are $(\phi,\varphi) \in \{(0,\pi), (\pi,0), (\phi, k -\phi), \text{with} \ \phi \in J_k\}$. The solutions to $\mathscr{E}_{k;-;-}$ are  
\begin{equation*}
S_{k;-;-} = 
\{ (0,4), (4,0), (2-2\cos(\phi),2-2\cos(k-\phi)), \phi \in J_k \},  
\end{equation*}
Let $f_{k;-;-}(\phi):= 2-2\cos(\phi) + 2-2\cos(k-\phi) = 4-4\cos(k/2)\cos(\phi-k/2)$. Thus
\begin{equation*}
\mathcal{E}_2(k) \supset E_{k;-;-} = \{4\} \cup f_{k;-;-}(J_k) = \begin{cases}
\{4\} \cup [4-4\cos(k/2),2-2\cos(k)] & \text{for} \ k \in (0,\pi) \\ 
\{4\} \cup [6+2\cos(k), 4-4\cos(k/2)] & \text{for} \ k \in (\pi,2\pi).
\end{cases}
\end{equation*} 
The solutions of $\mathscr{E}_{k;+;+}$ are 
\begin{equation*}
S_{k;+;+} = 
\{(0,4), (4,0), (2+2\cos(\phi),2+2\cos(k-\phi)), \phi \in J_k \}.
\end{equation*}
Let $f_{k;+;+}(\phi) := 2+2\cos(\phi) + 2+2\cos(k-\phi) = 4+4\cos(k/2)\cos(\phi-k/2)$. Then
\begin{equation*}
\mathcal{E}_2(k) \supset E_{k;+;+} = \{4\} \cup f_{k;+;+}(J_k) = \begin{cases}
\{4\} \cup [6+2\cos(k),4+4\cos(k/2)] & \text{for} \ k \in (0,\pi) \\
\{4\} \cup [4+4\cos(k/2),2-2\cos(k)] & \text{for} \ k \in (\pi,2\pi).
\end{cases}
\end{equation*} 
We now solve $\mathscr{E}_{k;-;+}$. With the same change of variables as before, this equation becomes 
\begin{equation}
\label{SLD2}
8\sin(k/2)\sin((\varphi-\phi+k)/2)\cos((\phi+\varphi)/2) =0.    
\end{equation}
Let $J'_k := [k,\pi]$ for $k \in (0,\pi)$ and $J'_k := [0,k-\pi]$ for $k \in(\pi,2\pi)$. The solutions to this equation are $(\phi,\varphi) \in \{(\phi,\pi-\phi), \ \text{with} \ \phi \in [0,\pi], (\phi, \phi-k), \ \text{with} \ \phi \in J'_k\}$. Thus
\begin{equation*}
S_{k;-;+} = \{(t,4-t), t \in [0,4], (2-2\cos(\phi),2-2\cos(k-\phi)), \phi \in J'_k \}.
\end{equation*}
Note that $f_{k;-;-}$ is strictly increasing on $J'_k$. Thus 
\begin{equation*}
\mathcal{E}_2(k) \supset E_{k;-;+} = \{4\} \cup f_{k;-;-}(J'_k) =
\begin{cases}
\{4\} \cup  [2-2\cos(k),6+2\cos(k)]  & \text{for} \ k \in (0,\pi) \\
\{4\} \cup  [2-2\cos(k),6+2\cos(k)]  & \text{for} \ k \in (\pi,2\pi). 
\end{cases}
\end{equation*}
Finally, by symmetry, $E_{k;+;-} = E_{k;-;+}$. Putting together our previous results, we have
\begin{equation}
\mathcal{E}_2(k) = [\lambda_{\ell}(k),\lambda_r(k)] = E_{k;-;-} \cup E_{k;+;+} \cup  E_{k;-;+} \cup E_{k;+;-}.
\end{equation}

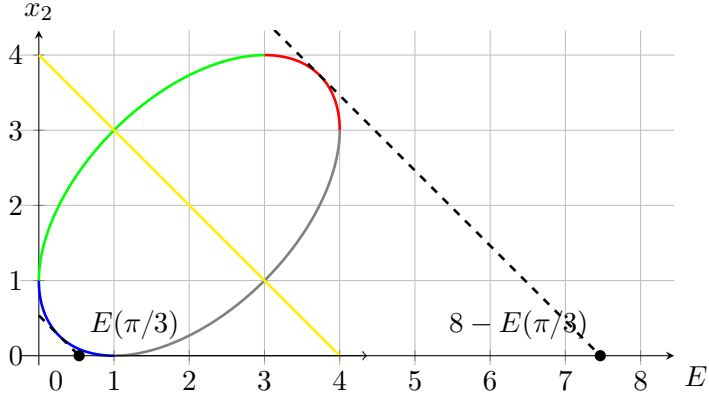
\begin{figure}
\begin{tikzpicture}[domain=0.01:9]
\begin{axis}
[grid = major, 
axis x line = middle, 
axis y line = middle, 
xlabel={$E$}, 
xlabel style={at=(current axis.right of origin), anchor=north west}, 
ylabel={$x_2$}, 
ylabel style={at=(current axis.above origin), anchor=south}, 
domain = 0.01:8, 
xmin = 0, 
xmax = 8.2, 
y=1cm,
x=1cm,
xtick={0,...,8},
ytick={0,1,2,3,4},
enlarge y limits={rel=0.03}, 
enlarge x limits={rel=0.03}, 
ymin = 0, 
ymax = 4.2, 
after end axis/.code={\path (axis cs:0,0) node [anchor=north west,yshift=-0.075cm] {0} node [anchor=east,xshift=-0.075cm] {0};}]
\addplot[color=blue, line width = 1.0pt, samples=400,domain=0:1.0472]({2-2*cos(deg(x))},{2-2*cos(deg(1.0472-x))});
\addplot[color=red, line width = 1.0pt, samples=400,domain=0:1.0472]({2+2*cos(deg(x))},{2+2*cos(deg(1.0472-x))});
\addplot[color=gray,line width = 1.0 pt, samples=400,domain=1.0472:3.14159] ({2-2*cos(deg(x))},{2-2*cos(deg(1.0472-x))});
\addplot[color=green, line width = 1.0pt, samples=400,domain=1.0472:3.14159]({2+2*cos(deg(x))},{2+2*cos(deg(1.0472-x))});
\addplot[color=black, dashed, line width = 1.0pt, samples=400,domain=0:4+4*cos(deg(0.5236))]({x},{4+4*cos(deg(0.5236))-x});
\addplot[color=black, dashed, line width = 1.0pt, samples=400,domain=0:4-4*cos(deg(0.5236))]({x},{4-4*cos(deg(0.5236))-x});
\addplot[color=yellow, line width = 1.0pt, samples=400,domain=0:4]({x},{4-x});
\node[anchor=west] (source) at (axis cs:0.5,0){};
\node (destination) at (axis cs:4.5,0){};
\draw[->](source)--(destination);
\node[label={270:{$x_1$}}] at (axis cs:4.5,0) {};

\node[label={85:{$E(\pi/3)$}},circle,fill,inner sep=1.5pt] at (axis cs:0.535898,0) {};
\node[label={100:{$8-E(\pi/3)$}},circle,fill,inner sep=1.5pt] at (axis cs:7.4641,0) {};
\end{axis}
\end{tikzpicture}
\caption{Solutions \textcolor{blue}{$S_{-;-}$}, \textcolor{red}{$S_{+;+}$}, \textcolor{gray}{$S_{-;+}$} and \textcolor{green}{$S_{+;-}$} to $\mathscr{E}_{k;\ast;\diamond}$ for $k=\pi/3$ and $\ast,\diamond \in \{-,+\}$, $d=2$}
\end{figure}

We now aim to derive \eqref{oksy} on $ \mathcal{L}_k := [0,E(k))$ for all $k \in \mathcal{T}$. Fix $k \in \mathcal{T}$ and $\lambda \in \mathcal{L}_k$. For $\lambda \in  \mathcal{L}_k$ define the function $F_{\lambda;k}$ on $[0,\lambda]$ by 
\begin{align*}
& F_{\lambda;k}(x) := \begin{cases}
h_{k;-}(x) + h_{k;-}(\lambda-x) & \text{for} \ (\lambda,k) \in  \mathcal{M}_k := \mathcal{L}_k \times (0,\pi) \\
h_{k;+}(x) + h_{k;+}(\lambda-x) & \text{for} \ (\lambda,k) \in  \mathcal{N}_k := \mathcal{L}_k \times (\pi,2\pi) 
\end{cases} \\
&= \begin{cases}
(\lambda-4)(\cos(k)-1) -\sin(k)\sqrt{x(4-x)} - \sin(k)\sqrt{(-x+\lambda)(4+x-\lambda)}  & \text{for} \ (\lambda,k) \in \mathcal{M}_k \\
(\lambda-4)(\cos(k)-1) +\sin(k)\sqrt{x(4-x)} + \sin(k)\sqrt{(-x+\lambda)(4+x-\lambda)}  & \text{for} \ (\lambda,k) \in \mathcal{N}_k.
\end{cases}
\end{align*}
A surprising calculation yields the single solution $x=\lambda/2$ to the equation $F_{\lambda;k}'(x) = 0$ for all $k \in \mathcal{T}$ and $\lambda \in \mathcal{L}_k$. Also, when $(\lambda,k) \in \mathcal{L}_k \times \mathcal{T}$, $F_{\lambda;k}(0) = F_{\lambda;k}(\lambda) > F_{\lambda;k}(\lambda/2)$. Hence 
\begin{equation*}
\forall \lambda \in \mathcal{L}_k, \ \forall k \in \mathcal{T}, \  \min_{x \in [0,\lambda]} F_{\lambda;k} (x) = F_{\lambda;k}(\lambda/2).   
\end{equation*}
Define for $\lambda \in \overline{\mathcal{L}_k} \times \mathcal{T}$ the function 
\begin{equation*}
f_k(\lambda) := F_{\lambda;k}(\lambda/2) = \begin{cases}
(\lambda-4)(\cos(k)-1)-\sin(k)\sqrt{\lambda(8-\lambda)} & \text{for} \ (\lambda,k) \in \mathcal{M}_k \\
(\lambda-4)(\cos(k)-1)+\sin(k)\sqrt{\lambda(8-\lambda)} & \text{for} \ (\lambda,k) \in \mathcal{N}_k.
\end{cases}
\end{equation*}
Then for all $k \in \mathcal{T}$, $f_k(0) = 4(1-\cos(k)) > 0$ and $f_k(E(k)) = 0$. We claim that for all $k \in \mathcal{T}$, $f_k$ is strictly decreasing and positive on $\mathcal{L}_k$. To prove this, consider the functions 
\begin{equation*}
m_{k;\mp}(\lambda):= (\lambda-4)(\cos(k)-1)\mp \sin(k)\sqrt{\lambda(8-\lambda)}    
\end{equation*} 
defined on $[0,8] \times \mathcal{T}$. The equation $m_{k;\ast}'(\lambda) = 0$ has a single solution 
\begin{equation*}
\lambda = 
4+4\sqrt{1-(1+\alpha(k)^2)^{-1}} > E(k) \quad \text{when} \ (k,\ast) \in (0,\pi)\times \{-\} \cup (\pi,2\pi) \times \{+\}.
\end{equation*}
Recall $\alpha(k) := (\cos(k)-1)(\sin(k))^{-1}$. The claim is therefore verified. Now let $E \in \mathcal{L}_k$, and choose $\epsilon' >0$ such that $\mathcal{I} := (E-\epsilon', E+\epsilon') \subset \mathcal{L}_k$. Recall that $R$ is the region defined after \eqref{oksy}. Let $(k,\ast) \in (0,\pi) \times \{-\} \cup (\pi,2\pi) \times \{+\}$. We have:
\begin{align*}
\inf \ \{g_{k;\ast}(x_1) + g_{k;\ast}(x_2) : (x_1,x_2) \in R\} &= \inf \ \{g_{k;\ast}(x) + g_{k;\ast}(\lambda-x) : \lambda \in \mathcal{I},x\in [0,\lambda]\} \\
&\geqslant \inf \ \{f_k(\lambda) + \lambda : \lambda \in \mathcal{I}\} \\
&\geqslant \epsilon + E -\epsilon'.
\end{align*}
Here $\epsilon$ is any real in $(0, f_k(E+\epsilon'))$. Taking $\epsilon'$ even smaller, the above inequalities remain valid with the same $\epsilon$ since $f_k$ is decreasing. Thus we may take $\epsilon' = \epsilon/2$ for example. Moreover, since $g_{k;+} \geqslant g_{k;-}(x)$ for all $(x,k) \in [0,4] \times (0,\pi)$ and $g_{k;-} \geqslant g_{k;+}(x)$ for all $(x,k) \in [0,4] \times (\pi,2\pi)$, we have proven that for all $(k,\ast,\diamond) \in \mathcal{T} \times \{-,+\} \times \{-,+\}$,
\begin{equation}
\label{hunbun}
\inf \ \{g_{k;\ast}(x_1) + g_{k;\diamond}(x_2) : (x_1,x_2) \in R\} \geqslant E+\epsilon/2.
\end{equation}
This proves \eqref{oksy} for $E \in \mathcal{L}_k$, with $\mathcal{I} = (E-\epsilon/2,E+\epsilon/2)$ and $k \in \mathcal{T}$. 

Now we proceed to extend the results for $d \geqslant 3$. Recall the properties of the function $g_{k;\pm}$ listed in Tables \ref{tab:g2} and \ref{tab:g4}. In particular, $g_{k;-}(x) \geqslant
0$ for all $(x,k) \in [0,\lambda_-(k)]\times (0,\pi)$ where $\lambda_-(k) = 2-2\cos(k)$, 
and $g_{k;+}(x) \geqslant
0$ for all $(x,k) \in [0,\lambda_+(k)]\times (\pi,2\pi)$ where $\lambda_+(k) = 2-2\cos(k)$. We take advantage of the fact that $E(k) < \lambda_-(k)$ for all $k\in (0,\pi)$ and $E(k) < \lambda_+(k)$ for all $k\in (\pi,2\pi)$. Again, let $E \in \mathcal{L}_k$, and choose $\epsilon >0$ such that $\mathcal{I} := (E-\epsilon/2,E+\epsilon/2) \subset \mathcal{L}_k$. Let $(k,\ast) \in (0,\pi) \times \{-\} \cup (\pi,2\pi) \times \{+\}$. Applying the two-dimensional result we obtain
\begin{align*}
\inf \ \Big\{ \sum_{i=1}^d g_{k}(x_i,y_i) : (x_i)_{i=1}^d \in R, (y_i)_{i=1}^d \in \{0,1\}^d \Big\}  & \geqslant \inf \ \Big\{ \sum_{i=1}^d g_{k; \ast}(x_i) : (x_i)_{i=1}^d \in R \Big\} \\
& \geqslant \inf \ \{ g_{k;\ast}(x) + g_{k;\ast}(\lambda-x) : \lambda \in \mathcal{I}, x \in [0,\lambda] \} \\
& \geqslant E+\epsilon/2.
\end{align*}
As this implies \eqref{oksy} for $E \in \mathcal{L}_k$, the proof is now complete.
\qed
\end{proof}
The method employed is optimal in the following sense: for $d=2$, for all $E \in [0,8] \setminus \mu(H) = [E(k),8-E(k)]$, and for all $\theta \in C_c^{\infty}(\R)$ with supp$(\theta) \ni E$, $\theta(\Delta)B_W\theta(\Delta)$ is not compact. Indeed, it is not too hard to see that if $\xi = (\xi_1,...,\xi_d) \in [-\pi,\pi]^d$ solves 
\begin{equation}
\label{guy98}
    \sum_{i=1}^d (2-2\cos(\xi_i)) = \sum_{i=1}^d (2-2\cos(\xi_i+k)), \quad \text{or} \quad \sum_{i=1}^d (2-2\cos(\xi_i)) = \sum_{i=1}^d (2-2\cos(\xi_i-k)),
\end{equation}
then $\theta(\Delta)B_W\theta(\Delta)$ is not compact for all $\theta$ with supp$(\theta) \ni E=\sum_i (2-2\cos(\xi_i))$. We note that \eqref{guy98} is precisely the same as \eqref{SLD1} and \eqref{SLD2} when $d=2$. By using the method of Lagrange multipliers for example, a slightly better value for $E(k)$ can be found when $d \geqslant 3$ (a value increasing with $d$). The method consists in extremizing $E = \sum_i (2-2\cos(\xi_i))$ with the constraints given in \eqref{guy98}. We move on to derive the classic Mourre estimate \eqref{Mourre K} for the full Schr{\"o}dinger operator $H'$. We really only have to show that $E_{\mathcal{I}}(H')B_{W'}E_{\mathcal{I}}(H')$ is compact. 
\begin{proposition}
\label{theProp33}
Let $k=(k_1,...,k_d) \in \mathcal{T}^d$ be the Wigner-von Neumann parameters, and let
\begin{equation}
E'(k) := \min \{ \ell(k_i) : 1\leqslant i \leqslant d \}, \quad \text{where} \quad \ell(k_i) := \begin{cases}
2-2\cos(k_i/2), \ \ k_i \in (0,2\pi/3] \\
2+2\cos(k_i), \ \ k_i \in (2\pi/3, \pi) \cup (\pi, 4\pi/3]\\
2+2\cos(k_i/2), \ \ k_i \in (4\pi/3,2\pi). 
\end{cases}
\label{dkdsl}
\end{equation}
Denote $\mu(H') := [0,E'(k)) \cup (4d-E'(k),4d]$. Then for every $E \in \mu(H')$ there exists $\epsilon(E) > 0$ such that for all $\theta \in C_c^{\infty}(\R)$ supported on $\mathcal{I} := (E - \epsilon, E + \epsilon)$, $\theta (\Delta) B_{W'} \theta (\Delta) = 0$. In particular, for every $E \in \mu(H')$, the classical Mourre estimate \eqref{Mourre K} holds for $H'$ on  $\mathcal{I}'$, where $\overline{\mathcal{I}'} \subset \mathcal{I}$.
\end{proposition}
\begin{proof}
As mentionned in Remark \ref{kol}, we show the result for $E \in [0,E'(k))$ and apply symmetry to get the result at the other end of the spectrum. We use the results from the one-dimensional case and follow the notation of Lemma \ref{ImportantLemma}. For now we denote by $\Delta$ the $1d$ Laplacian. The idea is the following : given $\lambda \in \sigma(\Delta) = [0,4]$, we want to find an interval $\mathcal{I}$ satisfying:
\begin{equation}
\begin{cases}
\mathcal{I} \ \text{is of the form} \ \mathcal{I}= [0,\lambda) \ \ \text{or} \ \ \mathcal{I}= (\lambda,4], \ \text{and} \\
\mathcal{I} \cap g_k(\mathcal{I},y) = \emptyset \ \ \text{for} \ \  y \in \{0,1\}.
\end{cases}
\label{alds}
\end{equation}
Here $g_k(x,y)$ is the function defined in \eqref{okid}. The motivation for wanting $\mathcal{I}$ of this form will be clear later in the proof. We examine the inequalities \eqref{a11}, \eqref{a12}, \eqref{a13} and \eqref{a14}. Fix $k \in (0,\pi)$. \eqref{a11} gives us \eqref{alds} for $\lambda \in [0,E_-(k))$ and $y=0$, whereas \eqref{a13} gives us \eqref{alds} for $\lambda \in [0,\lambda_+(k))$ and $y=1$, however with the condition that $\lambda < g_{k;+}(0)$.
We therefore let $\ell'(k) := \min(E_-(k),\lambda_+(k),g_{k;+}(0)) = \min(2-2\cos\left(k/2\right), 2+2\cos(k),2-2\cos(k))$, and it is readily checked that $\ell(k) = \ell'(k)$. Similarly, for $k \in (\pi,2\pi)$, we find $\ell(k) = \min(\lambda_-(k),E_+(k),g_{k;-}(0))=\min(2+2\cos(k),2+2\cos(k/2),2-2\cos(k))$. All intervals of the form $\mathcal{I} = [0,\lambda)$ with $\lambda<\ell(k)$ will satisfy \eqref{alds}.

Now we show how this can be of use for the two-dimensional case, although one can generalize for $d>2$. Let $k=(k_1,k_2)$ be the Wigner-von Neumann paramters and let $E'(k) := \min(\ell(k_1),\ell(k_2))$. Let $E \in \mathcal{L}_k := [0,E'(k))$ be given. Choose $\epsilon>0$ such that $\mathcal{I} := (E-\epsilon,E+\epsilon) \subset \mathcal{L}_k$. If $E=0$ was chosen, take $\mathcal{I} := [0,\epsilon) \subset \mathcal{L}_k$. Now let $\mathcal{I}_1 = \mathcal{I}_2 := [0, E+\epsilon)$. Notice that
\begin{equation}
\label{honneyd}
\{ (x_1,x_2) : x_1+x_2 \in \mathcal{I} \} \cap \left( \sigma(\Delta_1) \times \sigma(\Delta_2) \right) \subset \mathcal{I}_1 \times \mathcal{I}_2,
\end{equation} 
so that as functions on $(x_1,x_2) \in \sigma(\Delta_1)\times \sigma(\Delta_2)$, $\chi_\mathcal{I}(x_1+x_2) = \chi_{\mathcal{I}}(x_1+x_2) \chi_{\mathcal{I}_1}(x_1) \chi_{\mathcal{I}_2}(x_2)$. Thus as operators on $\ell^2(\Z) \otimes \ell^2(\Z)$, $E_\mathcal{I}(\Delta) = E_\mathcal{I}(\Delta)\cdot E_{\mathcal{I}_1} (\Delta_1) \otimes E_{\mathcal{I}_2} (\Delta_2)$. By \eqref{alds}, 
\begin{equation*}
\label{honey67}
E_{\mathcal{I}_i} (\Delta_i) \tilde{W'}_i E_{\mathcal{I}_i} (\Delta_i) = 0 \ \text{for} \ i=1,2.   
\end{equation*}
Recall that $B_{W_i'}$ is given by \eqref{BWd}. For $i=1,2$, 
\begin{equation*}
E_{\mathcal{I}_i} (\Delta_i) B_{W_i'} E_{\mathcal{I}_i} (\Delta_i) = E_{\mathcal{I}_i} (\Delta_i) \tilde{W'}_i E_{\mathcal{I}_i} (\Delta_i) (S_i^*-S_i) - (S_i^*-S_i) E_{\mathcal{I}_i} (\Delta_i) \tilde{W'}_i E_{\mathcal{I}_i} (\Delta_i) = 0.   
\end{equation*}
Therefore
\begin{align*}
\label{ahah38}
& E_\mathcal{I}(\Delta) \cdot B_{W_1'} \otimes W_2' \cdot E_\mathcal{I}(\Delta) \\ 
&= E_\mathcal{I}(\Delta) \cdot E_{\mathcal{I}_1} (\Delta_1) \otimes E_{\mathcal{I}_2} (\Delta_2) \cdot B_{W_1'} \otimes W_2' \cdot E_{\mathcal{I}_1} (\Delta_1) \otimes E_{\mathcal{I}_2} (\Delta_2) \cdot E_\mathcal{I}(\Delta) \\
&= E_\mathcal{I}(\Delta) \cdot E_{\mathcal{I}_1} (\Delta_1) B_{W_1'} E_{\mathcal{I}_1} (\Delta_1) \otimes E_{\mathcal{I}_2} (\Delta_2) W_2' E_{\mathcal{I}_2} (\Delta_2) \cdot E_\mathcal{I}(\Delta) \\ 
&= 0.
\end{align*}
Similarly, $E_\mathcal{I}(\Delta) \cdot W_1' \otimes B_{W_2'} \cdot E_\mathcal{I}(\Delta) = 0$. Thus $E_\mathcal{I}(\Delta) B_{W'}  E_\mathcal{I}(\Delta) = 0$, and the proof is complete.
\qed
\end{proof}


\section{Weighted Mourre Theory : Proof of Theorem \ref{lapy}}
\label{WMT}
In this section we prove the main result Theorem \ref{lapy}. For $s\in \R$, let $\langle N \rangle^{s}$ be the operator on $\ell_0(\Z^d)$ defined by $(\langle N \rangle^{s}u)(n) = \langle n \rangle^{s} u(n)$. The following Lemma says that the conjugate operator $A$ is comparable to the position operator $N$:

\begin{Lemma}
\label{Lemma of NA}
For all $\epsilon \in [0,1]$, both $\langle A \rangle^{\epsilon} \langle N \rangle^{-\epsilon}$ and $\langle N \rangle^{-\epsilon} \langle A \rangle^{\epsilon}$ are bounded operators.
\end{Lemma}
\proof
We use the notation $\| f \| \lesssim \| g \|$ if there is $c>0$ such that $\| f \| \leqslant c \| g \|$. Let $u \in \otimes _{i=1}^d \ell_0(\Z)$, which is dense in $\otimes _{i=1}^d \ell^2(\Z)$. We have: \\
\indent $\|\langle A \rangle u \|^2 = \|u\|^2+\|Au\|^2 \lesssim \|u\|^2 + \left( \sum_i \|u\| + \|N_i u\| \right) ^2 \lesssim \|u\|^2 +  \sum_i \big( \|u\|^2 + \|N_i u\|^2 \big) \lesssim \|\langle N \rangle u \|^2.$ \\
The first inequality follows from \eqref{generatorDilations}, and the second inequality holds by equivalence of the norms on $\ell^{1}(G)$ and $\ell^2(G)$ for finite dimensional Hilbert spaces $G$. By complex interpolation, $\|\langle A \rangle ^{\epsilon} u \| \lesssim \|\langle N \rangle ^{\epsilon} u \|$. Hence, for a dense set of $u' \in \otimes _{i=1}^d \ell^2(\Z)$, we have $\|\langle A \rangle ^{\epsilon} \langle N \rangle ^{-\epsilon} u' \| \lesssim \| u' \|$. This shows that $\langle A \rangle ^{\epsilon} \langle N \rangle ^{-\epsilon}$ extends to a bounded operator, and taking adjoints yields the result.
\qed

In our proof of the projected weighted Mourre estimate \eqref{projected weighted Mourre}, the following Lemma is crucial. At this point we will be using the full strength of hypothesis \eqref{kdkd5} on $V$, namely $\langle N \rangle ^{\rho}|V| \leqslant C$.

\begin{Lemma}
\label{sunnybunny68}
Let $\theta \in C_c^{\infty}(\R)$, and $\rho$ be as in \eqref{kdkd5}. Then for all $\epsilon \in [0,\min(\rho,1))$, the following operators are compact
\begin{equation}
\label{eqn:compact Difference}
(\theta(H)-\theta(\Delta))\langle A \rangle ^{\epsilon}, \quad \text{and} \quad  (\theta(H')-\theta(\Delta))\langle A \rangle ^{\epsilon}.
\end{equation}
\end{Lemma}
\proof
First, by Proposition \ref{goddam}, $\Delta \in \mathcal{C}^1(\langle A \rangle ^{\epsilon})$ since $f(x) = \langle x \rangle ^{\epsilon} \in \mathcal{S}^{\epsilon}$, thus $[\Delta,\langle A\rangle ^{\epsilon}]_{\circ}$ exists as a bounded operator. 
By the Helffer-Sj{\"o}strand formula and the resolvent identity, 
\begin{align*}
(\theta(H)-\theta(\Delta))\langle A \rangle ^{\epsilon} &= \frac{\i}{2\pi} \int_{\C} \frac{\partial \tilde{\theta}}{\partial \overline{z}} (z-H)^{-1}(W+V)(z-\Delta)^{-1} \langle A \rangle^{\epsilon} \dz \\ 
&= \frac{\i}{2\pi} \int _{\C} \frac{\partial \tilde{\theta}}{\partial \overline{z}} (z-H)^{-1}(W+V)\langle A \rangle^{\epsilon}(z-\Delta)^{-1} \dz \\
& \quad + \frac{\i}{2\pi} \int _{\C} \frac{\partial \tilde{\theta}}{\partial \overline{z}} (z-H)^{-1}(W+V) [(z-\Delta)^{-1},\langle A \rangle ^{\epsilon}]_{\circ}  \dz \\
&= \frac{\i}{2\pi} \int _{\C} \frac{\partial \tilde{\theta}}{\partial \overline{z}} (z-H)^{-1}(W+V)\langle N \rangle^{\epsilon}\langle N \rangle^{-\epsilon} \langle A \rangle^{\epsilon}(z-\Delta)^{-1} \dz \\
& \quad + \frac{\i}{2\pi} \int _{\C} \frac{\partial \tilde{\theta}}{\partial \overline{z}} (z-H)^{-1}(W+V) (z-\Delta)^{-1}[\Delta,\langle A \rangle ^{\epsilon}]_{\circ}(z-\Delta)^{-1}  \dz.
\end{align*}
By \eqref{Wiggy8}, $W$ and $W\langle N \rangle^{\epsilon}$ are compact, and so are $V$ and  $V\langle N \rangle^{\epsilon}$ by assumption \eqref{kdkd5}. By Lemma \ref{Lemma of NA}, $\langle N \rangle^{-\epsilon} \langle A \rangle^{\epsilon}$ is bounded, and so the integrands of the last two integrals are compact operators. With the support of $\theta$ compact, the integrals are converging in norm, and so the compactness of $(\theta(H)-\theta(\Delta))\langle A \rangle ^{\epsilon}$ is preserved in the limit. As for the Schr\"{o}dinger operator $H'$ the same proof works, but the additional point that has to be verified is that 
$W'\langle N \rangle^{\epsilon}$ is compact. Indeed, since
\begin{equation*}
\left(\prod_{i=1}^d q_i\sin(k_in_i)n_i^{-1} \right)^2 \langle n \rangle ^{2\epsilon} \leqslant \left(\prod_{i=1}^d q_i^2\sin^2(k_in_i)n_i^{-2} \right) \left( 1+\sum_{i=1}^d n_i^2 \right) \langle n \rangle ^{2(\epsilon-1)} \leqslant c \langle n \rangle ^{2(\epsilon-1)},
\end{equation*}
it follows that $W'(n) \langle n \rangle ^{\epsilon} \to 0$ as $|n|\to \infty$.
\qed

Because we are aiming at a projected Mourre estimate, we need some information on possible eigenvalues embedded in the interval on which the LAP takes place. Recall that $P$ denotes the orthogonal projection onto the pure point spectral subspace of $H$ (resp.\ $H'$), and $\mu(H)$ and $\mu(H')$ are points where the classical Mourre estimate hold for $H$ and $H'$ respectively.

\begin{Lemma} 
\label{hgod}
Let $E \in \mu(H)$ and suppose that $\ker (H-E) \in \mathcal{D}(A)$. Then there is an interval $\mathcal{I} \subset \mu(H)$ containing $E$ such that $PE_{\mathcal{I}}(H) $ and $P^{\perp}\eta(H) $ are of class $\mathcal{C}^1(A)$ for all $\eta \in C^{\infty}_c(\R)$ with supp$(\eta) = \mathcal{I}$. The corresponding statement also holds for $H'$.
\end{Lemma}
\proof
Since the Mourre estimate holds at $E$, the point spectrum is finite in a neighborhood $\mathcal{I}$ of $E$. Therefore $PE_{\mathcal{I}}(H)$ is a finite rank operator. Further shrinking $\mathcal{I}$ around $E$ if necessary, we have that $\ker (H-\lambda) \in \mathcal{D}(A)$ for all $\lambda \in \mathcal{I}$.  We may therefore apply Lemma \ref{rank one} to get $PE_{\mathcal{I}}(H)  \in \mathcal{C}^1(A)$. In addition, $P^{\perp}\eta (H)  = \eta(H) - P E_{\mathcal{I}}(H)\eta(H) \in \mathcal{C}^1(A)$.
\qed

We are now ready to prove the \emph{projected weighted Mourre estimate} \eqref{projected weighted Mourre}. The proof makes use of almost analytic extensions of $C^{\infty}(\R)$ bounded functions. The reader is invited to consult the appendix for some notation and useful results about these functions.

\begin{theorem} 
\label{guy56}
Let $E \in \mu(H)$ be such that $\ker (H-E) \subset \mathcal{D}(A)$. Then there exists an open interval $\mathcal{I} \ni E$ such that the projected weighted Mourre estimate \eqref{projected weighted Mourre} holds on $\mathcal{I}$ for all $s > 1/2$. Thus, for all compact $\mathcal{I}'$ with $\overline{\mathcal{I}'} \subset \mathcal{I}$, the LAP for $H$ holds with respect to $(\mathcal{I}',s,A)$. The corresponding result holds for $H'$.
\end{theorem}

\proof
First choose $\mathcal{I} \ni E$ so that for all $\lambda \in \mathcal{I}$, $\ker (H-\lambda) \in \mathcal{D}(A)$. This is of course possible as explained in Lemma \ref{hgod}. Let $\theta,\eta,\chi \in C^{\infty}_c (\mu(H))$ be bump functions such that $\eta \theta = \theta$, $\chi \eta = \eta$ and supp$(\chi) \subset \mathcal{I}$. Later we will shrink $\mathcal{I}$ appropriately. Let $s \in (1/2,2/3)$ be given. Define 
\begin{equation}
\varphi : \R \to \R, \qquad \varphi(t) := \int _{-\infty} ^t \langle x \rangle ^{-2s} dx.
\end{equation}
Note that $\varphi \in \mathcal{S}^{0}$ (see \eqref{decay1} for the definition of $\mathcal{S}^0$). For $R \geqslant 1$, consider the bounded operator 
\begin{align*}
F & := \Pp \tH [H, \i \varphi(A/R)]_{\circ} \tH \Pp \\
&= \const \int_{\C} \pp \Pp \tH \A [H, \i A/R]_{\circ} \A \tH \Pp \dz.
\end{align*}
By Lemma \ref{hgod}, $\Pp \eH \in C^1(A)$, so 
\begin{equation}
[\Pp \eH,\A]_{\circ} = \A [\Pp \eH, A/R]_{\circ} \A.
\end{equation} 
Next to $\Pp \tH$ we introduce $\Pp\eH$ and commute it with $\A$:
\begin{align*}
\label{la;}
F &= \const \int_{\C} \pp \Pp \tH \left(\A \Pp \eH + [\Pp \eH,\A]_{\circ} \right) [H, \i A/R]_{\circ} \\ 
& \ \ \ \ \ \ \ \ \ \ \ \ \ \ \ \ \ \ \ \ \ \ \ \ \left(\eH \Pp \A+ [\A,\Pp \eH]_{\circ} \right) \tH \Pp \dz \\
&= \const \int_{\C} \pp \Pp \tH \A \Pp \eH [H, \i A/R]_{\circ}  \\
& \ \ \ \ \ \ \ \ \ \ \ \ \ \ \ \ \ \ \ \ \ \ \ \ \ \ \ \ \ \ \ \ \ \ \ \ \ \ \ \ \ \ \ \ \eH \Pp \A \tH \Pp \dz + I_1 + I_2 + I_3 
\end{align*}
where $I_1$, $I_2$, $I_3$ are the 3 other integrals one obtains when expanding. For example 
\begin{align*}
I_1 &= \const \int _{\C} \pp \Pp \tH \A [\Pp\eH,A/R]_{\circ} \A [H, \i A/R ]_{\circ} \\
& \ \ \ \ \ \ \ \ \ \ \ \ \ \ \ \ \ \ \ \ \ \ \ \ \ \ \ \ \ \ \ \ \ \ \ \ \ \ \ \ \ \ \ \ \ \ \ \ \ \ \ \ \ \ \ \ \ \eH \Pp \A \tH \Pp \dz \\
&=  \Pp \tH \JapA ^{-s} \frac{B_1}{R^2} \JapA ^{-s} \tH \Pp
\end{align*}
for some bounded operator $B_1$ whose norm is uniformly bounded with respect to $R$, as shown in Lemma \ref{lemBox} with $\rho = 0$ and $n=3$. The same holds for $I_2$ and $I_3$, so for $i=1,2,3$,
\begin{equation*}
I_i = \Pp \tH \JapA ^{-s} \frac{B_i}{R^2} \JapA ^{-s} \tH \Pp.
\end{equation*}
Next to either $\eta(H)$ we insert $\chi(H)$, and we let $G:= \eH [H, \i A/R ]_{\circ} \eH$. We have:
\begin{align*}
F &= \const \int_{\C} \pp \Pp \tH \A \Pp \chi(H) G \chi(H) \Pp \A \tH \Pp \dz \\
&+ \Pp \tH \JapA ^{-s} \left(\frac{B_1+B_2+B_3}{R^2} \right) \JapA ^{-s} \tH \Pp.
\end{align*}
We decompose $G$ as follows 
\begin{align*}
G &= R^{-1} \bigg( \eD [\Delta, \i A]_{\circ} \eD  + \eD [W, \i A]_{\circ} \eD  +  \eD [V, \i A]_{\circ} \eD \\
&+ (\eH -\eD) [H, \i A]_{\circ} \eD + \eH [H, \i A ]_{\circ}(\eH -\eD) \bigg).
\end{align*}
We put into action our previous results. Shrink the support of $\eta$ if necessary to ensure that $\eD B_W \eD$ is compact (or zero) according to Lemma \ref{ImportantLemma} and Propositions \ref{sj9} and \ref{theProp33}. Thus $G=R^{-1}(\eD [\Delta, \i A]_{\circ} \eD + K_0)$ where $K_0:= \eD K_W \eD  + \eD B_W\eD \ + $
\begin{equation*}
+ \ \eD [V, \i A]_{\circ} \eD + (\eH -\eD) [H, \i A]_{\circ} \eD + \eH [H, \i A ]_{\circ}(\eH -\eD).
\end{equation*}
We claim that  
\begin{equation}
K_1 := \const \int_{\C} \pp \JapA^{s} \A \Pp \chi(H) K_0 \chi(H) \Pp \JapA^{s} \A  \dz,
\end{equation}
converges in norm to a compact operator for $s$ sufficiently close to $1/2$. Although $K_0$ is clearly compact, convergence in norm requires careful justification. Define
\begin{equation*}
K_{11}:= \langle A \rangle ^{\epsilon} \Pp \chi(H) \eD K_W \eD,  
\end{equation*}
\begin{equation*}
K_{12}:= \langle A \rangle ^{\epsilon} \Pp \chi(H) \eD B_W \eD,  
\end{equation*}
\begin{equation*}
K_{13}:= \langle A \rangle ^{\epsilon} \Pp \chi(H) \eD [V,\i A]_{\circ} \eD,
\end{equation*}
\begin{equation*}
K_{14}:= \langle A \rangle ^{\epsilon} \Pp  \chi(H) (\eta(H)-\eta(\Delta)) [H,\i A]_{\circ} \eD,  
\end{equation*}
\begin{equation*}
K_{15} := \eH [H,\i A]_{\circ} (\eta(H)-\eta(\Delta)) \chi(H) \Pp \langle A \rangle ^{\epsilon}.   
\end{equation*}
Let $\epsilon \in [0,\min(\rho,1))$. Since $\Pp \chi(H) \eD \in \mathcal{C}^1(A)$ and $f(x) = \langle x \rangle ^{\epsilon}  \in \mathcal{S}^{\epsilon}$, $[\langle A \rangle ^{\epsilon}, \Pp \chi(H) \eD]_{\circ}$ exists by Proposition \ref{goddam}. Moreover, $\langle N \rangle ^{\epsilon} K_W$ is compact and $\langle A \rangle ^{\epsilon}\langle N \rangle ^{-\epsilon}$ is bounded. Thus
\begin{equation*}
K_{11} = \Pp \chi(H) \eD \langle A \rangle ^{\epsilon} \langle N \rangle ^{-\epsilon} \langle N \rangle ^{\epsilon} K_W \eD  + [\langle A \rangle ^{\epsilon}, \Pp \chi(H) \eD]_{\circ}  K_W \eD
\end{equation*}
is compact. We turn to $K_{12}$. Commuting $\langle A \rangle ^{\epsilon}$ with $\Pp \chi(H)$ gives 
\begin{equation*}
K_{12} = \Pp \chi(H) \langle A \rangle ^{\epsilon} \langle N \rangle ^{-\epsilon} \langle N \rangle ^{\epsilon} \eD B_W \eD  + [\langle A \rangle ^{\epsilon}, \Pp \chi(H)]_{\circ}  \eD B_W \eD.
\end{equation*}
Applying the mean value theorem shows that $\langle N \rangle ^{\epsilon} [S_j, U_i]_{\circ}$ and $\langle N \rangle ^{\epsilon} [S_j^*, U_i]_{\circ}$ are compact  $\forall \ i,j =1,...,d$. Since $\eD B_W \eD = \sum_i [\eD, U_i]_{\circ} \tilde{W} (S^*_i-S_i) \eD - \eD (S^*_i-S_i)\tilde{W} [U_i,\eD]_{\circ}$ we see that $\langle N \rangle ^{\epsilon} \eD B_W \eD$, and hence $K_{12}$ is compact. As for $K_{13}$, we use the full strength of hypothesis \eqref{kdkd2} on $V$ to guarantee compactness of $\langle N \rangle ^{\epsilon} [V, \i A]_{\circ}$. Commuting $\langle A \rangle ^{\epsilon}$ with $\Pp \chi(H) \eD$ as before shows that $K_{13}$ is compact. By Lemma \ref{sunnybunny68}, $(\eta(H)-\eta(\Delta))\langle A \rangle ^{\epsilon}$ and its adjoint $\langle A \rangle ^{\epsilon}(\eta(H)-\eta(\Delta))$ are compact. Recall that this Lemma uses the full strength of hypothesis \eqref{kdkd5} on $V$. Commuting $\langle A \rangle ^{\epsilon}$ with $\Pp \chi(H)$ and using the fact that $[\Pp \chi(H), \langle A \rangle ^{\epsilon}]_{\circ}$ exists shows that $K_{14}$ and $K_{15}$ are compact.
Finally, $\langle A/R \rangle ^{\epsilon} \langle A \rangle ^{-\epsilon}$ and $\langle A \rangle ^{-\epsilon} \langle A/R \rangle ^{\epsilon}$ are uniformly bounded operators w.r.t. $R$. Thus invoking \eqref{dei} for $\ell=2$ and \eqref{didid2} we see that $K_1$ is a norm converging integral of compact operators provided $s$ additionally satisfies $s < 1/2 + \epsilon/2$. This proves the claim. 
Another important point to take into consideration is that 
\begin{equation}
\|K_1\| \leqslant C_1 \left( \|K_{11} \chi(H) \Pp\| + \|K_{12} \chi(H) \Pp\| + \| K_{13}\chi(H) \Pp\| + \|K_{14}  \chi(H) \Pp\| + \|\Pp \chi(H) K_{15}\| \right)
\end{equation}
for some finite $C_1>0$ independent of $R$. Hence $\|K_1\|$ vanishes as the support of $\chi$ gets tighter around $E$. Let 
\begin{equation*}
M:= \Pp \chi(H) \eD [\Delta, \i A]_{\circ} \eD \chi(H) \Pp. 
\end{equation*}
So far we have 
\begin{align*}
F &= \const \frac{1}{R} \int_{\C} \pp \Pp \tH \A M \A \tH \Pp \dz \\
&+  \Pp \tH \JapA ^{-s} \left( \frac{B_1+B_2+B_3}{R^2} + \frac{K_1}{R} \right) \JapA ^{-s} \tH \Pp.
\end{align*}
Next we commute $\A$ with $M$:
\begin{align*}
F &= \const \frac{1}{R} \int_{\C} \pp \Pp \tH (z-A/R)^{-2} M \tH \Pp \dz \\
&+ \const \frac{1}{R} \int_{\C} \pp \Pp \tH (z-A/R)^{-1} [M, \A]_{\circ} \tH \Pp \dz \\
&+  \Pp \tH \JapA ^{-s} \left( \frac{B_1+B_2+B_3}{R^2} + \frac{K_1}{R} \right) \JapA ^{-s} \tH \Pp.
\end{align*}
We apply \eqref{derivative} to the first integral (which converges in norm), while for the second integral we use the fact that $M \in \mathcal{C}^1(A)$ to conclude that there exists a uniformly bounded operator $B_4$ such that 
\begin{align*}
F &= R^{-1} \Pp \tH \varphi'(A/R) M \tH \Pp \\
&+  \Pp \tH \JapA ^{-s} \left( \frac{B_1+B_2+B_3+B_4}{R^2} + \frac{K_1}{R} \right) \JapA ^{-s} \tH \Pp.
\end{align*}
Now $\varphi'(A/R) = \langle A/R \rangle ^{-2s}$. As a result of the Helffer-Sj\"{o}strand formula, \eqref{dei} and \eqref{didid2},
\begin{equation*}
[\langle A/R \rangle ^{-s}, M ]_{\circ} \langle A/R \rangle ^{s} = R^{-1}B_5
\end{equation*} 
for some uniformly bounded operator $B_5$. Thus commuting $\langle A/R \rangle ^{-s}$ and $M$ gives
\begin{align*}
F &= R^{-1} \Pp \tH \JapA ^{-s} M \JapA ^{-s} \tH \Pp \\
&+  \Pp \tH \JapA ^{-s} \left( \frac{B_1+B_2+B_3+B_4+B_5}{R^2} + \frac{K_1}{R} \right) \JapA ^{-s} \tH \Pp \\
& \geqslant C R^{-1} \Pp \tH \JapA ^{-s} \Pp \chi(H) \eta^2(\Delta) \chi(H) \Pp \JapA ^{-s} \tH \Pp  \\
&+  \Pp \tH \JapA ^{-s} \left( \frac{B_1+B_2+B_3+B_4+B_5}{R^2} + \frac{K_1}{R} \right) \JapA ^{-s} \tH \Pp
\end{align*}
where $C>0$ comes from applying the Mourre estimate. Let
\begin{equation}
K_2:= \Pp \chi(H)(\eta^2(\Delta)-\eta^2(H))\chi(H) \Pp.
\end{equation} 
Note that $K_2$ is compact with $\|K_2\|$ vanishing as the support of $\chi$ gets tighter around $E$. Thus 
\begin{align*}
F & \geqslant C R^{-1} \Pp \tH \JapA ^{-s} \Pp \chi(H) \eta^2(H) \chi(H) \Pp \JapA ^{-s} \tH \Pp  \\
&+  \Pp \tH \JapA ^{-s} \left( \frac{B_1+B_2+B_3+B_4+B_5}{R^2} + \frac{K_1+K_2}{R} \right) \JapA ^{-s} \tH \Pp.
\end{align*}
Finally, we commute $\Pp \chi(H) \eta^2(H) \chi(H) \Pp = \Pp \eta^2(H)  \Pp$ with $\langle A/R \rangle ^{-s}$, and see that 
\begin{equation*}
[\Pp \eta^2(H) \Pp,\langle A/R \rangle ^{-s}]_{\circ} \langle A/R \rangle ^{s} = R^{-1}B_6
\end{equation*}
for some uniformly bounded operator $B_6$. Thus we have
\begin{align*}
F & \geqslant C R^{-1} \Pp \tH \JapA ^{-2s} \tH \Pp  \\
&+  \Pp \tH \JapA ^{-s} \left( \frac{B_1+B_2+B_3+B_4+B_5+B_6}{R^2} + \frac{K_1+K_2}{R} \right) \JapA ^{-s} \tH \Pp.
\end{align*}
To conclude, we shrink the support of $\chi$ to ensure that $\|K_1+K_2\| < C/3$ and choose $R \geqslant 1$ so that $\|\sum_{i=1}^6 B_i\|/R < C/3$. Then $K_1+K_2 \geqslant -C/3$ and $\sum_{i=1}^6 B_i/R \geqslant -C/3$, so 
\begin{equation}
F \geqslant \frac{C}{3R} \Pp \tH \JapA ^{-2s} \tH \Pp.
\end{equation}
Let $\mathcal{I}'$ be any open interval with $\overline{\mathcal{I}'} \subset \mathcal{I}$. Applying $E_{\mathcal{I}'}(H)$ on both sides of this inequality yields the projected weighted Mourre estimate \eqref{projected weighted Mourre}, with $c=C/(3R)$, $K=0$, and $s \in (1/2,\min(2/3,1/2+\rho/2))$. As a result of Theorem \ref{hoyaya}, the proof is complete.
\qed

\section{Appendix : Review of Almost Analytic Extenstions}
\label{appendix}
We refer to \cite{D}, \cite{DG}, \cite{GJ1}, \cite{GJ2}, \cite{HS} and \cite{M} for more details. We collect basic and essential results that are spread out in the mentionned literature. Let $\rho \in \R$ and denote by $\mathcal{S}^{\rho}(\R)$ the class of functions $\varphi$ in $C^{\infty}(\R)$ such that 
\begin{equation}
|\varphi^{(k)}  (x)| \leqslant C_k \langle x \rangle ^{\rho-k}, \quad \text{for all} \ k \geqslant 0.
\label{decay1}
\end{equation}
For $\rho <0$, $\mathcal{S}^{\rho}$ consists of the slowly decreasing functions at infinity, and contains every rational function whose denominator doesn't vanish on $\R$ and is of degree higher than its numerator. On the other hand, for $\rho >0$, $\mathcal{S}^{\rho}$ also allows for slowly increasing functions at infinity.

\begin{Lemma} \cite{D} and \cite{DG}
\label{DGDavies}
Let $\varphi \in \mathcal{S}^{\rho}$, $\rho \in \R$. Then for every $N \in \Z^+$ there exists a smooth function $\tilde{\varphi}_N : \C \to \C$, called an almost analytic extension of $\varphi$, satisfying:
\begin{equation}
\tilde{\varphi}_N(x+\i 0) = \varphi(x) \ \forall x \in \R;
\end{equation}
\begin{equation}
\mathrm{supp} \ (\tilde{\varphi}_N) \subset \{x+\i y : |y| \leqslant \langle x \rangle \};
\end{equation}
\begin{equation}
\tilde{\varphi}_N(x+\i y) = 0 \ \forall y \in \R \ \mathrm{whenever} \ \varphi(x) = 0;
\end{equation}
\begin{equation}
\forall \ell \in \N \cap [0,N], \Bigg| \frac{\partial \tilde{\varphi}_N}{\partial \overline{z}}(x+\i y) \Bigg| \leqslant c_{\ell} \langle x \rangle ^{\rho -1-\ell} |y|^{\ell} \ \mathrm{for \ some \ constants} \ c_{\ell} >0.
\label{dei}
\end{equation}
\end{Lemma}

\proof
Let $\theta \in C_c^{\infty}(\R)$ be a bump function such that $\theta(x) = 1$ for $x\in [-1/2,1/2]$ and $\theta(x) = 0$ for $x\in \R \setminus [-1,1]$, and consider 
\begin{equation}
\label{hocke}
\tilde{\varphi}_N(x+\i y) := \sum \limits_{n=0}^{N} \varphi^{(n)}(x) \frac{(\i y)^n}{n!}\theta \left( \frac{y}{\langle x \rangle}\right).
\end{equation}
The Wirtinger derivative is easily calculated:
\begin{equation*}
\frac{\partial \tilde{\varphi}_N}{\partial \overline{z}}(z) = \frac{1}{2}\sum \limits_{n=0}^{N} \frac{\varphi^{(n)}(x)}{\langle x \rangle} \frac{(\i y)^n}{n!} \theta ' \left( \frac{y}{\langle x \rangle}\right) \left(\i - \frac{yx}{\langle x \rangle^2}\right) \ + \ \frac{1}{2}\varphi^{(N+1)}(x)\frac{(\i y)^N}{N!} \theta \left( \frac{y}{\langle x \rangle}\right).
\end{equation*}
Therefore,
\begin{equation*}
\Bigg|\frac{\partial \tilde{\varphi}_N}{\partial \overline{z}}(z) \Bigg| \leqslant 
\sum \limits_{n=0}^{N} \frac{|\varphi^{(n)}(x)|}{\langle x \rangle} \frac{|y|^n}{n!} \chi_{\big\{\frac{\langle x \rangle}{2} \leqslant y \leqslant \langle x \rangle \big\}}(x,y) \ + \ \frac{1}{2}|\varphi^{(N+1)}(x)|\frac{| y|^N}{N!} \chi_{\big\{|y| \leqslant \langle x \rangle \big\}}(x,y).
\end{equation*}
It follows that:
\begin{align*}
\Bigg|\langle x \rangle ^{\ell +1 -\rho} |y|^{-\ell}\frac{\partial \tilde{\varphi}_N}{\partial \overline{z}}(x+\i y) \Bigg| & \leqslant 
\sum \limits_{n=0}^{\ell} C_{n} \frac{\langle x \rangle ^{\ell-n}}{n!} \left(\frac{\langle x \rangle}{2}\right)^{n-\ell} + \sum \limits_{n=\ell+1}^{N} C_{n} \frac{\langle x \rangle ^{\ell-n}}{n!} \langle x \rangle^{n-\ell} + \frac{1}{2} \frac{C_{N+1}}{N!}\\
& =\sum \limits_{n=0}^{\ell}  \frac{C_{n}}{n!} \frac{1}{2^{n-\ell}} + \sum \limits_{n=\ell+1}^{N} \frac{C_{n}}{n!} + \frac{1}{2} \frac{C_{N+1}}{N!} := c_{\ell}.
\end{align*}
\qed

Moreover, for $\varphi \in C_c^{\infty}(\R)$, we have the following key formula (cf.\ \cite{DG}):
\begin{equation*}
\varphi(t) = \frac{\i}{2\pi} \int_{\C} \frac{\partial \tilde{\varphi}_N}{\partial \overline{z}} (z) (z-t)^{-1} dz \wedge d\overline{z}, \ \ \forall N \in \Z^+.
\label{dei1}
\end{equation*}
By a limiting argument, this formula holds pointwise when $\varphi \in \mathcal{S}^{\rho}, \rho < 0$. Now let $A$ be a self-adjoint operator acting on a Hilbert space $\mathcal{H}$. In terms of operators, we have
\begin{equation}
\varphi(A) = \frac{\i}{2\pi} \int_{\C} \frac{\partial \tilde{\varphi}_N}{\partial \overline{z}} (z) (z-A)^{-1} dz \wedge d\overline{z}.
\label{dei22}
\end{equation}
Thus, in the case where $\varphi \in \mathcal{S}^{\rho},\rho < 0$, the point of the analytic extension is that it allows for an explicit expression of the operator $\varphi(A)$ whose existence is known from the spectral theorem. This formula can be extended for $\rho \geqslant 0$ as follows: 
\begin{Lemma} \cite{GJ1}
\label{extend rho}
Let $\rho \geqslant 0$ and $\varphi \in \mathcal{S}^{\rho}$. Let $\varphi(A)$ with domain $\mathcal{D}(\varphi(A)) \supset \mathcal{D}(\langle A \rangle ^{\rho})$ be the operator whose existence is assured by the spectral theorem. Then for $f \in \mathcal{D}(\langle A \rangle ^{\rho})$, 
\begin{equation}
\varphi(A)f = \lim \limits_{R \to \infty} \frac{\i}{2\pi} \int_{\C} \frac{\partial (\tilde{\varphi\theta_R})_N}{\partial \overline{z}} (z) (z-A)^{-1} f dz \wedge d\overline{z},
\label{gross}
\end{equation}
where $\theta_R(x) := \theta(x/R)$ and $\theta$ is like in Lemma \ref{DGDavies}. 
\end{Lemma}
\proof
\begin{equation*}
\frac{\i}{2\pi} \int_{\C} \frac{\partial (\tilde{\varphi\theta_R})_N}{\partial \overline{z}} (z) (z-A)^{-1} f dz \wedge d\overline{z} = (\varphi\theta_R)(A)f = (\varphi_{\rho} \theta_R)(A) \langle A \rangle ^{\rho} f,
\end{equation*}
where $\varphi_{\rho}(t) := \varphi (t) \langle t \rangle^{-\rho}$ is a bounded function. Thus $(\varphi_{\rho} \theta_R)(A)$ is converging strongly to $\varphi(A) \langle A \rangle ^{-\rho}$, and this shows \eqref{gross}.
\qed

Notice that when $\rho<0$, the r.h.s. of \eqref{gross} is equal to the r.h.s of \eqref{dei22} applied to $f$ by the dominated convergence theorem. 
\begin{Lemma}
Let $\rho < 0$ and $\varphi \in \mathcal{S}^{\rho}$. Then for all $k \in \N$ and $N \in \N$:
\begin{equation}
\label{derivative}
\varphi^{(k)}(A) = \frac{\i (k!)}{2\pi} \int _{\C} \frac{\partial \tilde{\varphi}_N}{\partial \overline{z}} (z) (z-A)^{-1-k} dz \wedge d\overline{z}
\end{equation}
where the integral exists in the norm topology. For $\rho \geqslant 0$, the following limit exists:
\begin{equation}
\varphi^{(k)}(A)f = \lim \limits_{R \to \infty} \frac{\i (k!)}{2\pi} \int_{\C} \frac{\partial (\tilde{\varphi\theta_R})_N}{\partial \overline{z}} (z) (z-A)^{-1-k} f dz \wedge d\overline{z}, \quad \text{for all} \  f \in \mathcal{D}(\langle A \rangle ^{\rho}).
\label{gross2}
\end{equation}
In particular, if $\varphi \in \mathcal{S}^{\rho}$ with $0 \leqslant \rho < k$ and $\varphi^{(k)}$ is a bounded function, then $\varphi^{(k)}(A)$ is a bounded operator and \eqref{derivative} holds (with the integral converging in norm).
\end{Lemma}
\proof
First we show \eqref{derivative}. Assume for now that $\varphi \in C_c ^{\infty}(\R)$. By definition,
\begin{equation*}
\varphi^{(k)}(A) = \frac{\i}{2\pi} \int _{\C} \frac{\partial \tilde{\varphi^{(k)}}_N}{\partial \overline{z}} (z) (z-A)^{-1} dz \wedge d\overline{z}.
\end{equation*}
Now consider $\tilde{\varphi^{(k)}}_N$ and the $k^{\text{th}}$ partial derivative of $\tilde{\varphi}_N$ in $x$ respectively given by
\begin{align*}
\tilde{\varphi^{(k)}}_N (x+ \i y) &= \sum \limits_{n=0}^{N} \varphi^{(k+n)}(x) \frac{(\i y)^n}{n!}\theta \left( \frac{y}{\langle x \rangle}\right), \ \text{and} \\
\partial_x ^{k} \tilde{\varphi}_N (x+ \i y) &= \sum \limits_{n=0}^{N} \varphi^{(k+n)}(x) \frac{(\i y)^n}{n!}\theta \left( \frac{y}{\langle x \rangle}\right) + \sum \limits_{n=0}^{N} \frac{(\i y)^n}{n!} \sum \limits_{j=1} ^{k} \frac{k!}{j!(k-j)!} \varphi^{(n+k-j)}(x) \partial_x ^{j} \theta \left( \frac{y}{\langle x \rangle}\right).
\end{align*}
Notice that $|\tilde{\varphi^{(k)}}_N (x+ \i y) - \partial_x ^{k} \tilde{\varphi}_N (x+ \i y)|$ is identically zero in a small strip around the x-axis, and so by \cite[Lemma 2.2.3]{D}, we have that 
\begin{equation*}
\varphi^{(k)}(A) = \frac{\i}{2\pi} \int _{\C} \frac{\partial \tilde{\varphi^{(k)}}_N}{\partial \overline{z}} (z) (z-A)^{-1} dz \wedge d\overline{z} = \frac{\i}{2\pi} \int _{\C} \frac{\partial \partial_x ^{k} \tilde{\varphi_N}}{\partial \overline{z}} (z) (z-A)^{-1} dz \wedge d\overline{z}.
\end{equation*}
The result follows by performing $k$ partial integrations w.r.t. $x$. The formula extends to $\varphi \in \mathcal{S}^{\rho}$ by density of $C_c^{\infty}(\R)$ in $\mathcal{S}^{\rho}$ for $\rho < 0$. As for \eqref{gross2}, let $\phi_{\rho}(t):= \langle t \rangle ^{-{\rho}}$. We have, using \eqref{derivative},
\begin{equation*}
\frac{\i (k!)}{2\pi} \int_{\C} \frac{\partial (\tilde{\varphi\theta_R})_N}{\partial \overline{z}} (z) (z-A)^{-1-k} f dz \wedge d\overline{z} = (\varphi\theta_R)^{(k)}(A)f = \left(\sum \limits_{j=0} ^{k} c_j \varphi ^{(k-j)} \phi_{\rho} (\theta_R)^{(j)}\right)(A) \langle A \rangle ^{\rho} f.
\end{equation*}
Here $c_j := k! (j!(k-j)!)^{-1}$. First note that $(\theta_R)^{(j)}(x) = R^{-j}\theta^{(j)}(x/R)$. Moreover, $\varphi^{(k-j)} \phi_{\rho}$ are bounded functions for $0 \leqslant j \leqslant k$, so $(\varphi^{(k-j)} \phi_{\rho})(A)$ are bounded operators and \begin{equation*}
(\varphi^{(k-j)} \phi_{\rho})(A) = \slim \limits_{R \to \infty} (\varphi^{(k-j)}\phi_{\rho})(A) \theta^{(j)}(A/R). \end{equation*}
Thus 
\begin{equation*}
\slim \limits_{R \to \infty}   \left(\sum \limits_{j=1} ^{k} c_j \varphi ^{(k-j)} \phi_{\rho} (\theta_R)^{(j)}\right)(A) = 0
\end{equation*} 
and this implies \eqref{gross2}. Finally, if $0 \leqslant \rho < k$ and $\varphi^{(k)}$ is a bounded function, then we use \eqref{dei} with $\ell = k+1$ and apply the dominated convergence theorem to pass the limit inside the integral.  
\qed

\begin{Lemma} \cite{GJ2}
Let $s \in [0,1]$ and $D := \{ (x,y) \in \R^2 : 0 < |y| \leqslant \langle x \rangle \}$. Then there exists $c > 0$ independent of $A$ such that for all $z = x+\i y \in D$ :
\begin{equation}
\|\langle A \rangle ^s (A-z)^{-1}\| \leqslant c \cdot \langle x \rangle ^s \cdot |y|^{-1}.
\label{didid2}
\end{equation}
\end{Lemma}\begin{Lemma}
\label{lemBox}
Let $\varphi \in \mathcal{S}^{\rho}$, and let $B_1,...,B_n$ be bounded operators. Then for $s \in [0,1]$ satisfying $s < 1 -(1+\rho)/n$, and any $N \geqslant n$, the following integral
\begin{equation}
\label{kdi}
\const \int_{\C} \frac{\partial \tilde{\varphi}_N}{\partial \overline{z}} \prod\limits_{i=1}^{n} \langle A \rangle ^s (z-A)^{-1} B_i \ \dz
\end{equation}
converges in norm to a bounded operator. In particular, for $\rho =0$ and $n \geqslant 3$, \eqref{kdi} converges to a bounded operator for $s \in [0,2/3)$.
\end{Lemma}
\proof
Combine \eqref{didid2} and \eqref{dei} for $\ell=n$.
\qed

We end this section with two very useful formulas. \begin{proposition} \cite{GJ1} Let $T$ be a bounded self-adjoint operator satisfying $T \in C^1(A)$. Then:
\begin{equation}
\label{use1}
[T,(z-A)^{-1}]_{\circ} = (z-A)^{-1}[T,A]_{\circ}(z-A)^{-1},
\end{equation}
and for any $\varphi \in \mathcal{S}^{\rho}$ with $\rho < 1$, $T \in \mathcal{C}^1(\varphi(A))$ and
\begin{equation}
\label{use2}
[T,\varphi(A)]_{\circ} = \const \int_{\C} \frac{\partial \tilde{\varphi}_N}{\partial \overline{z}} (z-A)^{-1}[T,A]_{\circ}(z-A)^{-1} \dz.
\end{equation}
\label{goddam}
\end{proposition}


\end{document}